\documentclass[11pt]{article}

\usepackage[margin=1in]{geometry}
\usepackage[scaled]{helvet}  

\usepackage[utf8]{inputenc}
\usepackage{amsmath}
\usepackage{amsfonts}
\usepackage{epsfig, graphics, graphicx}
\usepackage{program,color}
\usepackage[usenames,dvipsnames,svgnames,table]{xcolor} 
\usepackage[square,numbers]{natbib}
\usepackage[english]{babel}
\usepackage{algorithmicx}
\usepackage[ruled]{algorithm}
\usepackage{algpseudocode}
\usepackage{caption}
\usepackage{subcaption}
\usepackage{array}

\newtheorem{theorem}{Theorem}
\newtheorem{claim}{Claim}

\newcommand{\half}{\tfrac1{2}}
\newcommand{\E}{\mathbb{E}}

\newcommand{\tr}{*}
\newcommand{\rows}{m}
\newcommand{\cols}{n}
\newcommand{\bX}{{\boldsymbol{X}}}
\newcommand{\bU}{{\boldsymbol{U}}}
\newcommand{\bS}{{\boldsymbol{S}}}
\newcommand{\bV}{{\boldsymbol{V}}}
\newcommand{\bK}{{\boldsymbol{K}}}

\newcommand{\Id}{{\boldsymbol{I}}}
\newcommand{\bSigma}{{\boldsymbol{\Sigma}}}
\newcommand{\by}{{\boldsymbol{y}}}
\newcommand{\bz}{{\boldsymbol{z}}}
\newcommand{\be}{{\boldsymbol{e}}}

\newcommand{\bA}{{\boldsymbol{A}}}
\newcommand{\bb}{{\boldsymbol{b}}}
\newcommand{\bx}{{\boldsymbol{x}}}

\newcommand{\bbeta}{{\boldsymbol{\beta}}}
\newcommand{\balpha}{{\boldsymbol{\alpha}}}

\newcommand{\betaopt}{{\boldsymbol{\beta^{\circ}}}}

\newcommand\TS{\rule{0pt}{4ex}}       
\newcommand\BS{\rule[-3ex]{0pt}{0pt}} 
\newcommand\ddfrac[2]{\frac{\displaystyle #1}{\displaystyle #2}}

\title{Rows vs. Columns: Randomized Kaczmarz or Gauss-Seidel for Ridge Regression} 

\author{Ahmed Hefny\footnote{Authors are listed in alphabetical order.} \\
Machine Learning Department \\
Carnegie Mellon University
\and Deanna Needell\footnotemark[1] \\
Department of Mathematical Sciences \\
Claremont McKenna College
\and Aaditya Ramdas\footnotemark[1] \\
Departments of EECS and Statistics \\
University of California, Berkeley
}  

\date{\today}

\begin{document}
\maketitle

\begin{abstract}
The Kaczmarz and Gauss-Seidel methods aim to solve a $\rows \times \cols$ linear system $\bX\bbeta = \by$ by iteratively refining the solution estimate; the former uses random rows of $\bX$ 
{to update $\bbeta$ given the corresponding equations} and the latter uses random columns of $\bX$ {to update corresponding coordinates in $\bbeta$}.  
Recent work analyzed these algorithms in a parallel comparison for the overcomplete and undercomplete systems, showing convergence to the ordinary least squares (OLS) solution and the minimum Euclidean norm solution respectively. 
This paper considers the natural follow-up to the OLS problem --- ridge or Tikhonov regularized regression. By viewing them as variants of randomized coordinate descent, we present variants of the randomized Kaczmarz (RK) and randomized Gauss-Siedel (RGS) for solving this system and derive their convergence rates.  
We prove that a recent proposal, which can be interpreted as randomizing over \textit{both} rows and columns, is strictly suboptimal --- instead, one should always work with randomly selected \textit{columns} (RGS) when $\rows > \cols$ (\#rows $>$ \#cols) and with randomly selected \textit{rows} (RK) when $ \cols > \rows$ (\#cols $>$ \#rows).
\end{abstract}

\section{Introduction}

Solving systems of linear equations $\bX \bbeta = \by$, also sometimes called ordinary least squares (OLS) regression, dates back to the times of Gauss, who introduced what we now know as Gaussian elimination. A widely used iterative approach to solving linear systems is the conjugate gradient method. However, our focus is on variants of a recently popular class of randomized algorithms whose revival was sparked by \citet{SV09:Randomized-Kaczmarz}, who proved linear convergence  of the Randomized Kaczmarz (RK) algorithm that works on the \textit{rows} of $\bX$ (data points).  \citet{LL10:Randomized-Methods} afterwards proved linear convergence of Randomized Gauss-Seidel (RGS), which instead operates on the \textit{columns} of $\bX$ (features).  Recently, \citet{MaConvergence15} provided a side-by-side analysis of RK and RGS for linear systems in a variety of under- and over-constrained settings. Indeed, both RK and RGS can be viewed in two dual ways --- either as variants of stochastic gradient descent (SGD) for minimizing an appropriate objective function, or as variants of randomized coordinate descent (RCD) on an appropriate linear system --- and to avoid confusion, and aligning with recent literature, we refer to the row-based variant as RK and the column-based variant as RGS for the rest of this paper. The advantage of such approaches is that they do not need access to the entire system but rather only individual rows (or columns) at a time.  This makes them amenable in data streaming settings or when the system is so large-scale that it may not even load into memory.  This paper only analyzes variants of these types of iterative approaches, showing in what regimes one may be preferable to the other.

For statistical as well as computational reasons, one often prefers to solve what is called \textit{ridge} regression or \textit{Tikhonov-regularized} least squares regression. This corresponds to solving the convex optimization problem 
\[ 
\min_\bbeta \|\by - \bX \bbeta \|_2^2 + \lambda \|\bbeta\|^2
\]
 for a given parameter $\lambda$ (which we assume is fixed and known for this paper) and
 a (real or complex) $\rows\times \cols$ matrix $\bX$. We can reduce the above problem to solving a linear system, in \textit{two} different (dual) ways.
 Our analysis will show that the performance of the two associated algorithms is very different when $\rows < \cols$ and when $\rows > \cols$.  

\paragraph{Contribution}

There exist a large number of algorithms, iterative and not, randomized and not, for this problem. In this work, we will only be concerned with the aforementioned subclass of randomized algorithms, RK and RGS.  
Our main contribution is to analyze 
the convergence rates of variants of RK and RGS for ridge regression, showing linear convergence in expectation, but the emphasis will be on the convergence rate parameters (e.g. condition number) that come into play for our algorithms when $m>n$ and $m<n$.   We show that when $m>n$, one should randomize over columns (RGS), and when $m<n$, one should randomized over rows (RK).  Lastly, we show that randomizing over both rows and columns (effectively proposed in prior work as the \textit{augmented projection method} by \citet{ivanov2013kaczmarz}) leads to wasted computation and is always suboptimal.  


\paragraph{Paper Outline}

In Section~\ref{sec:algs}, we introduce the two most relevant (for our paper) algorithms for linear systems, Randomized Kaczmarz (RK) and Randomized Gauss-Seidel (RGS). In Section~\ref{sec:ours} we describe our proposed variants of these algorithms and provide a simple side-by-side analysis in various settings.
In Section~\ref{sec:variants}, we briefly describe what happens when RK or RGS is naively applied to the ridge regression problem.  We also describe a recent proposal to tackle this issue, which can  be interpreted as a combination of an RK-like and RGS-like updates, and discuss its drawbacks compared to our solution.  We conclude in Section~\ref{sec:exp} with detailed experiments that validate the derived theory.

\section{Randomized Algorithms for OLS}\label{sec:algs}


We begin by briefly describing the randomized Kaczmarz and Gauss-Siedel methods, which serve as the foundation to our approach for ridge regression.

\paragraph{Notation} Throughout the paper we will consider an $m\times n$ (real or complex) matrix $\bX$ and write $\bX^i$ to represent the $i$th row of $\bX$ (or $i$th entry of a vector) and $\bX_{(j)}$ to denote the $j$th column.  We will write solution estimations $\bbeta$ as column vectors.  We write vectors and matrices in boldface, and constants in standard font.  The singular values of a matrix $\bX$ are written as $\sigma(\bX)$ or just $\sigma$, with subscripts $\min$, $\max$ or integer values corresponding to the smallest, largest, and numerically ordered (increasing) values.  We denote the identity matrix by $\Id$, with a subscript denoting the dimension when needed.  Unless otherwise specified, the norm $\|\cdot\|$ denotes the standard Euclidean norm (or spectral norm for matrices).  We use the norm notation $\|\bz\|_{\bA^*\bA}^2$ to mean $\langle\bz, \bA^*\bA\bz\rangle = \|\bA\bz\|^2$.  
We will use the superscript $\circ$ to denote the optimal value of a vector.

\subsection{Randomized Kaczmarz (RK) for $\bX \bbeta = \by$}

The Kaczmarz method \cite{Kac37:Angenaeherte-Aufloesung} is also known in the tomography setting as the \textit{Algebraic Reconstruction Technique} (ART) \cite{GBH70:Algebraic-Reconstruction,Nat01:Mathematics-Computerized,Byr08:Applied-Iterative,herman2009fundamentals}.  In its original form, each iteration consists of projecting the current estimate onto the solution space given by a single row, in a cyclic fashion.  It has long been observed that selecting the rows $i$ in a random fashion improves the algorithm's performance, reducing the possibility of slow convergence due to adversarial or unfortunate row orderings \cite{HS78:Angles-Null,HM93:Algebraic-Reconstruction}.  

Recently, \citet{SV09:Randomized-Kaczmarz} showed that the RK method converges linearly to the solution $\betaopt$ of $\bX\bbeta = \by$ in expectation, with a rate that depends on natural geometric properties of the system, improving upon previous convergence analyses (e.g. \cite{XZ02:Method-Alternating}).  In particular, they propose the variant of the Kaczmarz update with the following selection strategy: 

\begin{equation}\label{eq:rk}
\bbeta_{t+1} := \bbeta_t + \frac{(y^i - \bX^{i}\bbeta_t)}{\|\bX^i\|^2_2} (\bX^i)^*, \quad\text{where}\quad \Pr(\text{row} = i) = \frac{\|\bX^i\|^2_2}{\|\bX\|_F^2},
\end{equation}
where the first estimation $\bbeta_0$ is chosen arbitrarily and $\|\bX\|_F$ denotes the Frobenius norm of $\bX$.
\citet{SV09:Randomized-Kaczmarz} then prove that the iterates $\bbeta_t$ of this method satisfy the following,
\begin{equation}\label{RVbound}
\mathbb{E}\|\bbeta_t - \betaopt\|_2^2 \leq \left(1 - \frac{\sigma^2_{\min}(\bX)}{\|\bX\|_F^2}\right)^t \|\bbeta_0 - \betaopt\|_2^2,
\end{equation}
where we refer to the quantity $\frac{\sigma^2_{\min}(\bX)}{\|\bX\|_F^2}$ as the scaled condition number of $\bX$.
This result was extended to the inconsistent case \cite{Nee10:Randomized-Kaczmarz}, derived via a probabilistic almost-sure convergence perspective \cite{CP12:Almost-Sure-Convergence}, accelerated in multiple ways \cite{Elf80:Block-Iterative-Methods,EN11:Acceleration-Randomized,popa2012kaczmarz,NW12:Two-Subspace-Projection,needell2013paved}, and generalized to other settings \cite{LL10:Randomized-Methods,richtarik2012iteration,NSWjournal}. 

%
%
%
%

\subsection{Randomized Gauss-Seidel (RGS) for $\bX \bbeta = \by$}

The Randomized Gauss-Seidel (RGS) method \footnote{Sometimes this method is referred to in the literature as Randomized Coordinate Descent (RCD).
However, we establish in 
Section \ref{sec:RK_RGS_RCD} that both Randomized Kaczmarz and Randomized Gauss-Seidel can be viewed as 
randomized coordinate descent methods on different variables.}
selects columns rather than rows in each iteration.  For a selected coordinate $j$, RGS attempts to minimize the objective function $L(\bbeta) = \half \|\by-\bX\bbeta\|^2_2$ with respect to coordinate $j$ in that iteration.  It can thus be similarly defined by the following update rule
:  
\begin{equation}\label{eq:rgs}
\bbeta_{t+1} := \bbeta_t + \frac{\bX_{(j)}^*(\by-\bX\bbeta_t)}{\|\bX_{(j)}\|^2_2} \be_{(j)},\quad\text{where}\quad \Pr(\text{col} = j) = \frac{\|\bX_{(j)}\|^2_2}{\|\bX\|_F^2},
\end{equation}
where $\be_{(j)}$ is the $j$th coordinate basis column vector (all zeros with a $1$ in the $j$th position). 
\citet{LL10:Randomized-Methods} showed that the residuals of RGS converge again at a linear rate, 
\begin{equation}\label{LLbound}
\mathbb{E}\|\bX\bbeta_t - \bX\betaopt\|_2^2 \leq \left(1 - \frac{\sigma^2_{\min}(\bX)}{\|\bX\|_F^2}\right)^t \|\bX\bbeta_0 - \bX\betaopt\|_2^2.
\end{equation}

Of course when $\rows > \cols$ and the system is full-rank, this convergence also implies convergence of the iterates $\bbeta_t$ to the solution $\betaopt$.  Connections between the analysis and performance of RK and RGS were recently studied in \cite{MaConvergence15}, which also analyzed extended variants to the Kacmarz \cite{ZF12:Randomized-Extended} and Gauss-Siedel method \cite{MaConvergence15} which always converge to the least-squares solution in both the underdetermined and overdetermined cases.
Analyses of RGS usually applies more generally than our OLS problem, see e.g. \citet{nesterov2012efficiency} or \citet{richtarik2012iteration} for further details. 

\subsection{RK and RGS as variants of randomized coordinate descent}
\label{sec:RK_RGS_RCD}

Both RK and RGS can be viewed in the following fashion: suppose we have a positive definite matrix $\bA$, and we want to solve $\bA\bx=\bb$.  
Casting the solution to the linear system as the solution to $\min_{\bx} \half \bx^\tr\bA\bx - \bb^\tr\bx$, one can derive the coordinate-descent update 
\[
\bx_{t+1} = \bx_t + \frac{b_i-\bA^i\bx_t}{A_{ii}} \be_{(i)},
\]
where $b_i-\bA^i \bx_t$ is basically the $i$-th coordinate of the gradient, and $A_{ii}$ is the Lipschitz constant of the $i$-th coordinate of the gradient (see related works e.g. \citet{LL10:Randomized-Methods}, \citet{Nesterov12}, \citet{RicTak12}, \citet{SidLee13}). In this light, the original RK update in Equation \eqref{eq:rk} can be seen as the randomized coordinate descent rule for the positive semidefinite system $\bX\bX^\tr \balpha = \by$ (using the standard primal-dual mapping $\bbeta = \bX^\tr \balpha$) and treating $\bA = \bX\bX^\tr$ and $\bb=\by$. Similarly, the RGS update in \eqref{eq:rgs} can be seen as the randomized coordinate descent rule for  the positive semidefinite system $\bX^\tr \bX \bbeta = \bX^\tr \by$ and treating $\bA = \bX^\tr \bX$ and $\bb=\bX^\tr \by$. 
\footnote{It is worth noting that both methods also have an interpretation as variants of the stochastic gradient method (for minimizing a different objective function from above, namely $\min_\bbeta \|\by - \bX \bbeta\|_2^2$), but this connection will not be further pursued here.}

\section{Variants of RK and RGS for Ridge Regression}\label{sec:ours}

Utilizing the connection to coordinate descent, one can derive two algorithms for ridge regression, depending on how we formulate the linear system that solves the underlying optimization problem. In the first formulation, we let $\betaopt$ be the solution of the system
\begin{equation}\label{eq:syseq}
(\bX^\tr \bX + \lambda \Id_\cols)\beta = \bX^\tr \by,
\end{equation}
and we attempt to solve for $\betaopt$ iteratively by updating an initial guess $\bbeta_0$ using columns of $\bX$.
In the second, we note the identity $\betaopt = \bX^\tr \balpha^\circ$, where $\balpha^\circ$ is the optimal solution of the system
\begin{equation}\label{eq:syseq2}
 (\bX \bX^\tr + \lambda \Id_\rows)\balpha = \by,
\end{equation}
and we attempt to solve for $\balpha^\circ$ iteratively by updating an initial guess $\balpha_0$ using rows of $\bX$.
The formulations \eqref{eq:syseq} and \eqref{eq:syseq2} can be viewed as primal and dual variants, respectively, of the ridge regression problem, and our analysis is related to the analysis in \cite{csiba2016coordinate} (independent work around the same time as this current work).

It is a short exercise to verify that the optimal solutions of these two seemingly different formulations are actually identical (in the machine learning literature, the latter is simply an instance of kernel ridge regression). The second method's randomized coordinate descent updates are:
\begin{eqnarray}
\delta^i_t &=& \frac{y^i - \bbeta_t^\tr(\bX^{i})^* - \lambda \alpha^i_t}{\|\bX^i\|^2 + \lambda}, \label{eq:kacz1}\\
\alpha^i_{t+1} &=& \alpha^i_t + \delta^i_t,  \label{eq:kacz2}
\end{eqnarray}
where the $i$th row is selected with probability proportional to $\|\bX^i\|^2 + \lambda$. We may keep track of $\bbeta$ as $\balpha$ changes using the update $\bbeta_{t+1} = \bbeta_t + \delta^i_t (\bX^{i})^\tr$.
Denoting $\bK := \bX \bX^\tr$ as the Gram matrix of inner products between rows of $\bX$, and $r^i_t := y_i - \sum_j K_{ij} \alpha_t^j$ as the $i$th residual at step $t$, the above update for $\balpha$ can be rewritten as
\begin{eqnarray}\label{eq:kerkacz1}
\alpha^i_{t+1} &=& \frac{K_{ii}}{K_{ii}+\lambda}\alpha^i_t + \frac{y_i - \sum_j K_{ij} \alpha_t^j}{K_{ii} + \lambda}\\
&=& S_{\frac{\lambda}{K_{ii}}} \left( \alpha_t^i + \frac{r^i_t}{K_{ii}} \right)\label{eq:kerkacz3}
\end{eqnarray}
where row $i$ is picked with probability proportional to $K_{ii}+\lambda$  and $S_a(z) := \frac{z}{1+a}$.

In contrast, we write below the randomized coordinate descent updates for the first linear system. Analogously calling ${\bf r}_t := \by - \bX \bbeta_t$ as the residual vector, we  have
\begin{eqnarray}
\beta^j_{t+1} &=&\beta^j_t + \frac{\bX_{(j)}^\tr \by - \bX_{(j)}^\tr \bX \bbeta_t -\lambda \beta^j_t}{\|\bX_{(j)}\|^2 + \lambda}\label{eq:RGSridge1}\\
&=& S_{\frac{\lambda}{\|\bX_{(j)}\|^2}} \left(\beta^j_t + \frac{\bX_{(j)}^\tr {\bf r}_t}{\|\bX_{(j)}\|^2} \right). \label{eq:RGSridge3} 
\end{eqnarray}
Next, we analyze the difference in these approaches, equation \eqref{eq:kerkacz3} being called the RK update (working on rows) and equation \eqref{eq:RGSridge3} being called the RGS update (working on columns).

%

\subsection{Computation and Convergence}
The algorithms presented in this paper are of computational interest because they completely avoid inverting, storing or even forming  $\bX\bX^\tr$ and $\bX^\tr\bX$.  
The RGS updates take $O(\rows)$ time, since each column (feature) is of size $\rows$. In contrast, the RK updates take $O(\cols)$ time since that is the length of a row (data point). 
While the RK and RGS algorithms are similar and related, one should not be tempted into thinking their convergence rates are the same. Indeed, using a similar style proof as presented in \cite{MaConvergence15}, one can analyze the convergence rates in parallel as follows. Let us denote 
\[
\bSigma' := \bX^\tr \bX + \lambda \Id_\cols \in \mathbb{R}^{n \times n} 
\quad\text{and}\quad \bK' := \bX \bX^\tr + \lambda \Id_m \in \mathbb{R}^{m \times m}
\] 
for brevity, and let $\sigma_1,\sigma_2,...$ be the singular values of $\bX$ in increasing order.  Observe that
$$
\sigma_{\min}(\Sigma') = 
\begin{cases}  \sigma_1^2 + \lambda ~\mbox{~if $\rows>\cols$} \\
\lambda ~\mbox{~if $\cols>\rows$}
\end{cases}
\quad\text{and}\quad
\sigma_{\min}(K') = 
\begin{cases}  \lambda ~\mbox{~if $\rows>\cols$} \\
\sigma_1^2 + \lambda ~\mbox{~if $\cols>\rows$}.
\end{cases}
$$
Then, denoting $\betaopt$ and $\balpha^\circ$ as the solutions to the two ridge regression formulations, and $\bbeta_0$ and $\balpha_0$ as the initializations of the two algorithms, we can prove the following result. 
\begin{theorem}
The rate of convergence for RK for ridge regression is :
\begin{eqnarray}
\E\|\balpha_t - \balpha^\circ\|_{\bK+\lambda \Id_\cols}^2 &\leq& 
\begin{cases} \left( 1 - \ddfrac{\lambda}{\sum_i \sigma_i^2 + \rows\lambda}\right)^{t} \|\balpha_0 - \balpha^\circ\|_{\bK+\lambda \Id_\rows}^2 ~\mbox{~if $\rows>\cols$ }\\
\left( 1 - \ddfrac{\sigma_1^2 + \lambda}{\sum_i \sigma_i^2 + \rows\lambda}\right)^{t} \|\balpha_0 - \balpha^\circ\|_{\bK+\lambda \Id_\rows}^2 ~\mbox{~if $\cols>\rows$} .
\end{cases}
\end{eqnarray}
The rate of convergence for RGS for ridge regression is :
\begin{eqnarray}
\E\|\bbeta_t - \betaopt\|_{\bX^\tr\bX +\lambda \Id_\cols}^2 &\leq& 
\begin{cases} \left( 1 - \ddfrac{\sigma_1^2 + \lambda}{\sum_i \sigma_i^2 + \cols\lambda}\right)^{t} \|\bbeta_0 - \betaopt\|_{\bX^\tr \bX+\lambda \Id_\cols}^2 ~\mbox{~if $\rows>\cols$} \\
\left( 1 - \ddfrac{ \lambda}{\sum_i \sigma_i^2 + \cols\lambda}\right)^{t} \|\bbeta_0 - \betaopt\|_{\bX^\tr\bX+\lambda \Id_\cols}^2 ~\mbox{~if $\cols>\rows$}.
\end{cases}
\end{eqnarray}
\label{thm:rates}
\end{theorem}
Before we present a proof, we may immediately note that  RGS  is preferable in the overdetermined case while RK is preferable in the underdetermined case.  
Hence, our  proposal for solving such systems as as follows: 
\begin{quote}
When $m>n$, always use RGS, and when $m<n$, always use RK. 
\end{quote}

\begin{proof}
We summarize the proof in the table below, which allows for a side-by-side comparison of how the error reduces by a constant factor after one step of the algorithm. Let $\E_t$ represent the expectation with respect to the random choice (of row or column) in iteration $t$, conditioning on all the previous iterations.

\begin{table}[h!]
\begin{tabular}{|l|l|} 
\hline
\TS \BS
 RK: $\E_t \|\balpha_{t+1} - \balpha^\circ\|_{\bK'}^2$ 
 & 
 RGS: $\E_t \|\bbeta_{t+1} - \betaopt\|_{\bSigma'}^2$
\\
\hline
\TS \BS
$\stackrel{(i)}{=} \E_t \left(\|\balpha_t - \balpha^\circ\|_{\bK'}^2 - \|\balpha_{t+1} - \balpha_t\|_{\bK'}^2\right)$
 & 
 $\stackrel{(i')}{=}\E_t \left(\|\bbeta_t - \betaopt\|_{\bSigma'}^2 - \|\bbeta_{t+1} - \bbeta_t\|_{\bSigma'}^2\right)$
 \\
\hline
\TS \BS
$\stackrel{(ii)}{=}\|\balpha_t - \balpha^\circ\|_{\bK'}^2 - \sum_i \frac{\bK_{ii} + \lambda }{\|\bX\|_F^2 + \rows \lambda} \frac{(y_i - \sum_{j}\bK_{ij}\alpha^j_t - \lambda\alpha^i_t)^2}{\bK_{ii} + \lambda}$
& 
$\stackrel{(ii')}{=}\|\bbeta_t - \betaopt\|_{\bSigma'}^2 - \sum_j \frac{\|\bX_{(j)}\|^2 + \lambda }{\|\bX\|_F^2 + \cols \lambda} \frac{( \bX_{(j)}^\tr ( \by - \bX\bbeta_t) - \lambda \beta_t^j)^2}{\|\bX_{(j)}\|^2 + \lambda} $
\\
\hline
\TS \BS
$= \|\balpha_t - \balpha^\circ\|_{\bK'}^2 - \frac{\| (\by - \bK' \balpha_t)\|^2}{\|\bX\|_F^2 + \rows\lambda}$
& 
$= \|\bbeta_t - \betaopt\|_{\bSigma'}^2 - \frac{\| \bX^\tr \by - \bSigma'\bbeta_t\|^2}{\|\bX\|_F^2 + \cols\lambda}$
\\
\hline
\TS \BS
$= \|\balpha_t - \balpha^\circ\|_{\bK'}^2 - \frac{\| \bK'(\balpha^\circ - \balpha_t)\|^2}{\|\bX\|_F^2 + \rows\lambda}$
& 
$= \|\bbeta_t - \betaopt\|_{\bSigma'}^2 - \frac{\| \bSigma'(\betaopt - \bbeta_t)\|^2}{\|\bX\|_F^2 + \cols\lambda}$
\\
\hline
\TS \BS
$\leq \|\balpha_t - \balpha^\circ\|_{\bK'}^2 -\frac{\sigma_{\min}(\bK')\| \balpha^\circ - \balpha_t\|_{\bK'}^2}{Tr(\bK')}$
& 
$\leq \|\bbeta_t - \betaopt\|_{\bSigma'}^2 -\frac{\sigma_{\min}(\bSigma')\| \betaopt - \bbeta_t\|_{\bSigma'}^2}{Tr(\bSigma')}$
\\
\hline
\TS \BS
$=\begin{cases} \left( 1 - \frac{\lambda}{\sum_i \sigma_i^2 + \rows\lambda}\right) \|\balpha_t - \balpha^\circ\|_{\bK'}^2 ~\mbox{~if $\rows>\cols$ }\\
\left( 1 - \frac{\sigma_1^2 + \lambda}{\sum_i \sigma_i^2 + \rows\lambda}\right) \|\balpha_t - \balpha^\circ\|_{\bK'}^2 ~\mbox{~if $\cols>\rows$} 
\end{cases}$
& 
$=\begin{cases} \left( 1 - \frac{\sigma_1^2 + \lambda}{\sum_i \sigma_i^2 + \cols\lambda}\right) \|\bbeta_t - \betaopt\|_{\bSigma'}^2 ~\mbox{~if $\rows>\cols$} \\
\left( 1 - \frac{ \lambda}{\sum_i \sigma_i^2 + \cols\lambda}\right) \|\bbeta_t - \betaopt\|_{\bSigma'}^2 ~\mbox{~if $\cols>\rows$}
\end{cases}$
\\
\hline
\end{tabular}
\end{table}
Applying these bounds recursively, we obtain the theorem. We must only justify the equalities $(i),(i')$ and $(ii),(ii')$ at the start of the above proof. We derive these here for $\betaopt$ and it holds with a similar derivation for $\balpha^\circ$. For succinctness, denote $\bbeta_t$ as simply $\bbeta$ and $\bbeta_{t+1}$ as $\bbeta^+$.

$(i),(i')$ hold simply because of the optimality conditions for $\balpha^\circ$ and $\betaopt$.   Note that $(i')$ can be interpreted as an instance of Pythagoras' theorem, and to verify its truth we must prove the orthogonality of $\bbeta^+ - \betaopt$ and $\bbeta^+ - \bbeta$ under the inner-product induced by $\bSigma'$. In other words, we need to verify that
\[
 (\bbeta^+ - \bbeta)^\tr (\bX^\tr \bX + \lambda \Id_n) (\bbeta^+ - \betaopt) = 0
\]
Since $\bbeta^+ - \bbeta$ is parallel to the $j$-th coordinate vector ${\bf e}_j$, and since $(\bX^\tr \bX + \lambda \Id_n) \betaopt = \bX^\tr \by$, this amounts to verifying that
\[
{\bf e}_j^\tr (\bX^\tr \bX + \lambda \Id_n) \bbeta^+ -  \bX_{(j)}^\tr \by = 0, \text{ or equivalently } \bX_{(j)}^\tr \bX \bbeta^+ + \lambda \bbeta^+_j -  \bX_{(j)}^\tr \by = 0.
\]
This is simply a restatement of 
\[
\frac{\partial}{\partial \bbeta_j}\left[\|\by - \bX \bbeta \|_2^2 + \lambda \|\bbeta\|^2 \right] = 0,
\] which is how the update for $\bbeta_j$ is derived. 

Similarly, $(ii),(ii')$ hold simply because of the definition of the procedure that updates $\bbeta^+$ from $\bbeta$, and $\balpha^+$ from $\balpha$. To see this for $\bbeta$, note that $\bbeta^+$ and $\betaopt$ differ in the $j$-th coordinate with probability $\frac{\|\bX_{(j)}\|^2 + \lambda}{\|\bX\|_F^2 + n\lambda}$, and hence 
\[
\E_t\|\bbeta^+ - \bbeta\|_{\bSigma'}^2 = \sum_j \frac{\|\bX_{(j)}\|^2 + \lambda}{\|\bX\|_F^2 + n\lambda} (\bbeta_j^+ - \bbeta_j)\bSigma'_{jj}(\bbeta_j^+ - \bbeta_j)
\]
Now, note that $\bSigma'_{jj} = \|\bX_{(j)}\|^2 + \lambda$. Then, as an immediate consequence of update \eqref{eq:RGSridge1}, we obtain the equality $(ii')$.

This concludes the proof the theorem. \hfill$\square$
\end{proof}

\section{Suboptimal RK/RGS Algorithms for Ridge Regression}\label{sec:variants}



One can  view  $\bbeta_{RR}$ and $\balpha_{RR}$ simply as solutions to the two linear systems 
$$(\bX^\tr \bX + \lambda \Id_n)\bbeta = \bX^\tr \by \quad\text{and}\quad 
(\bX \bX^\tr + \lambda \Id_m) \balpha = \by.$$ 
If we naively use RK or RGS on either of these systems (treating them as solving $\bf{A} x=\bf{b}$ for some given $\bA$ and $\bf b$), then we may apply the bounds \eqref{RVbound} and \eqref{LLbound} to the matrix $\bX^\tr\bX + \lambda\Id_n$ or $\bX\bX^\tr + \lambda\Id_m$.  This, however, yields a bound on the convergence rate which depends on the \textit{squared} scaled condition number of $\bX^\tr\bX + \lambda\Id$, which is approximately the \textit{fourth power} of the scaled condition number of $\bX$.  This dependence is suboptimal, so much so that it becomes highly impractical to solve large scale problems using these methods.  This is of course not surprising since this naive solution does not utilize any structure of the ridge regression problem (for example, the aforementioned matrices are positive definite).  One thus searches for more tailored approaches --- indeed,
our proposed RK and RGS updates whose computation are still only $O(\cols)$ or $O(\rows)$ per iteration and yield linear convergence with dependence only on the scaled condition number of $\bX^\tr\bX + \lambda\Id_n$ or $\bX\bX^\tr + \lambda\Id_m$, and not their square.  The aforementioned updates and their convergence rates are motivated by a clear understanding of how RK and RGS methods relate to each other as in \cite{MaConvergence15} and jointly to positive semi-definite systems of equations. However, these are not the only options that one has. Indeed, we may wonder if an algorithm that suitably randomizes over \textit{both} rows and columns 

\subsection{Augmented Projection Method (IZ)}

We now describe a creative proposition by Ivanov and Zhdanov \cite{ivanov2013kaczmarz}, which we refer to as the \textit{augmented projection method} or \textit{Ivanov-Zhdanov method} (IZ).
We consider the regularized normal equations of the system \eqref{eq:syseq}, as demonstrated in \cite{zhdanov2012method,ivanov2013kaczmarz}. Here, the authors recognize that the solution to the system \eqref{eq:syseq} can be given by
\[
\left( 
\begin{array}{cc}
\sqrt{\lambda} \Id_m & \bX\\
\bX^\tr & -\sqrt \lambda \Id_n
\end{array}
\right)
\left( 
\begin{array}{c}
\balpha' \\
\bbeta
\end{array}
\right)
=
\left( 
\begin{array}{c}
\by \\
\bf 0_n 
\end{array}
\right).
\]
Here we use $\balpha'$ to differentiate this variable from $\balpha$, the variable involved in the ``dual'' system $(\bK+\lambda \Id_m) \balpha = y$ --- note that $\balpha'$ and $\balpha$ just differ by a constant factor $\sqrt{\lambda}$. The authors propose to solve the system \eqref{eq:syseq} by applying the Kaczmarz algorithm (and in their experiments, they apply Randomized Kacmarz) to the aforementioned system.
As they mention, the advantage of rewriting it in this fashion is that the condition number of the $(m+n)\times (m+n)$ matrix
$$
\bf{A} := \left( 
\begin{array}{cc}
\sqrt{\lambda} \Id_m & \bX\\
\bX^\tr & -\sqrt \lambda \Id_n
\end{array}
\right)
$$
is the square-root of the condition number of the $n \times n$ matrix $\bX^\tr \bX + \lambda \Id_n$. Hence, the RK algorithm on the aforementioned system converges an order of magnitude faster than running RK on \eqref{eq:syseq} using the matrix $\bX^\tr \bX + \lambda \Id_n$.

\subsection{The IZ algorithm effectively randomizes over \textit{both} rows and columns}

Let us look at what the IZ algorithm does in more detail.  
The two sets of equations are:
\begin{align}
\sqrt \lambda \balpha' + \bX \bbeta = \by \quad\text{and}\quad 
\bX^\tr \balpha' = \sqrt \lambda \bbeta.
\label{eq:iz_cond}
\end{align}
First note that the first $m$ rows of $\bA$ correspond to rows of $\bX$ and have a squared norm $\|\bX^i\|^2 + \lambda$ and the next $n$ rows of $\bA$ correspond to columns of $\bX$ and have a norm $\|\bX_{(j)}\|^2 + \lambda$. Hence, $\|{\bA}\|_F^2 = 2\|\bX\|_F^2 + (m+n)\lambda$. 

This means one can interpret picking a random row of the $(m+n)\times (m+n)$ matrix $\bA$ (with probability proportional to its row norm) as a two step process. If one of the first $m$ rows of $\bA$ are picked, we are effectively doing a ``row update'' on $\bX$. If one of the last $n$ rows of $\bA$ are picked, we are effectively doing a ``column update'' on $\bX$. Hence, we are effectively randomly choosing between doing ``row updates'' or ``column updates'' on $\bf X$ (choosing to do a row update with probability $\frac{\|\bX\|_F^2 + m\lambda}{2\|\bX\|_F^2 + (m+n)\lambda}$ and a column update otherwise). 

If we choose to do ``row updates'', we then choose a random row of $\bf X$ (with probability proportional to $\frac{\|\bX^i\|^2 + \lambda}{\|\bX\|_F^2 + m\lambda}$ as done by RK). If we choose to do ``column updates'', we then choose a random column of $\bf X$ (with probability proportional to $\frac{\|\bX_{(j)}\|^2 + \lambda}{\|\bX\|_F^2 + n\lambda}$ as done by RGS).

If one selects a random row $i \leq \rows$ with probability proportional to $\|\bX^i\|^2 + \lambda$, the equation we greedily satisfy is
$$
\sqrt \lambda \be_{(i)}^\tr \balpha' + \bX^i \bbeta = y^i
$$
using the update 
\begin{align}
(\balpha'_{t+1}, \bbeta_{t+1}) = (\balpha'_t,\bbeta_t) + \frac{ y^i - \sqrt \lambda \be_{(i)}^\tr \balpha'_t - \bX^i \bbeta_t}{\|\bX^i\|^2 + \lambda}(\sqrt \lambda \be_{(i)}, \bX^i),
\label{eq:iz_row}
\end{align}
which can be computed in $O(\rows+\cols)$ time.  Similarly, if a random column $j\leq \cols$ is selected with probability proportional to $\|\bX_{(j)}\|^2 + \lambda$, the equation we greedily satisfy is
$$
\bX_{(j)}^\tr \balpha' = \sqrt{\lambda} \be_{(j)}^\tr \bbeta
$$
with the update in $O(\rows+\cols)$ time of
\begin{align}
(\balpha_{t+1}', \bbeta_{t+1}) = (\balpha_t',\bbeta_t) + \frac{ \sqrt{\lambda} \be_{(j)}^\tr \bbeta_t - \bX_{(j)}^\tr \balpha'_t}{\|\bX_{(j)}\|^2 + \lambda}(\bX_{(j)}^\tr, -\sqrt \lambda \be_{(j)}).
\label{eq:iz_col}
\end{align}

Next, we further study the behavior of this method under different initialization conditions. 

\subsection{The Wasted Iterations of the IZ Algorithm} 

The augmented projection method attempts to find $\balpha'$ and $\bbeta$
that satisfy conditions \eqref{eq:iz_cond}.
It is insightful to examine the behavior of that approach when one of these conditions is already satisfied. 

\begin{claim}\label{clm:IZ1}
Assume $\balpha'_0$ and $\bbeta_0$ are initialized such that
$$
\bbeta_0 = \frac{\bX^\tr \balpha_0'}{\sqrt{\lambda}}.
$$
(for example, all zeros). Then:
\begin{enumerate}
\item The update equation \eqref{eq:iz_row} 
is an RK-style update on $\balpha$.
\item The condition $\bbeta_t = \frac{\bX^\tr \balpha_t'}{\sqrt{\lambda}}$ is automatically maintained for all $t$.
\item Update equation \eqref{eq:iz_col} has absolutely no effect.
\end{enumerate}
\end{claim}

\begin{proof}
Suppose at some iteration $\bbeta_t = \frac{\bX^\tr \balpha'_t}{\sqrt{\lambda}}$ holds. Then assuming we do a row update, substituting this into \eqref{eq:iz_row} gives, for the $i$th variable being updated,
$$
\alpha'^i_{t+1} = \alpha'^i_t + \frac{y^i \sqrt{\lambda} - \lambda \alpha'^i_t - \bX^i \bX^\tr \balpha' }{\|\bX^i\|^2 + \lambda} = \frac{\|\bX^i\|^2}{\|\bX^i\|^2 + \lambda} \alpha'^i_t
+ \frac{y^i \sqrt{\lambda} - \bX^i \bX^\tr \balpha'}{\|\bX^i\|^2 + \lambda},
$$
which (as we will later see in more detail) can be viewed as an RK-style update on $\balpha$.
The parallel update to $\bbeta$ can then be rewritten as
$$
\bbeta_{t+1} = \bbeta_t + \frac{y^i - \sqrt{\lambda} \alpha'^i_t - \bX^i \bX^\tr \balpha' }{\|\bX^i\|^2 + \lambda} \bX^i  = \bbeta_t + \frac{\alpha'^i_{t+1} - \alpha'^i_t}{\sqrt{\lambda}} \bX^i,
$$
which automatically keeps condition $\bbeta = \frac{\bX^\tr \balpha'}{\sqrt{\lambda}}$ satisfied.
Since this condition is already satisfied, we have that $\sqrt{\lambda} \be_{(j)}^\tr \bbeta_t = \bX_{(j)}^\tr \balpha'_t$.  Thus, if we then run any column update from
 \eqref{eq:iz_col}, the numerator of the additive term is zero and we get $$(\balpha'_{t+1}, \bbeta_{t+1}) = (\balpha'_t, \bbeta_t).$$ 
\hfill$\square$
\end{proof}


\begin{claim}\label{clm:IZ2}
Assume $\balpha'_0$ and $\bbeta_0$ are initialized such that
$$
\balpha_0' = \frac{\by - \bX\bbeta_0}{\sqrt{\lambda}}.
$$
(for example, $\bbeta_0$ is zero, $\balpha'_0 = \by/\sqrt \lambda$). Then:
\begin{enumerate}
\item The update equation \eqref{eq:iz_col} 
is an RGS-style update on $\bbeta$.
\item The condition $\balpha_t' = \frac{\by - \bX\bbeta_t}{\sqrt{\lambda}}$ is automatically maintained for all $t$.
\item Update equation \eqref{eq:iz_row} has absolutely no effect.
\end{enumerate}
\end{claim}
\begin{proof}
Suppose at some iteration  $\balpha'_t = \frac{\by - \bX\bbeta_t}{\sqrt{\lambda}}$ holds. Then assuming we do a column update, substituting this in \eqref{eq:iz_col} gives, for the $j$th variable being updated
$$
\beta_{t+1}^j = \beta_t^j + \frac{\sqrt{\lambda} {\bX^\tr_{(j)}} \balpha'_t - \lambda \beta_t^j}{\|\bX_{(j)}\|^2 + \lambda} = \beta_t^j + \frac{{\bX^\tr_{(j)}} (\by - \bX \bbeta_t) - \lambda \beta_t^j}{\|\bX_{(j)}\|^2 + \lambda},
$$
which (as we will later see in more detail) is an RGS-style update. The parallel update on 
$\balpha'$ can then be rewritten as
$$
\balpha'_{t+1} = \balpha'_t - \frac{{\bX^\tr_{(j)}} \balpha'_t - \sqrt{\lambda} \beta_t^j}{\|\bX_{(j)}\|^2 + \lambda} \bX^\tr_{(j)} = \balpha'_t - \frac{\beta_{t+1}^j - \beta_t^j}{\sqrt{\lambda}} {\bX^\tr_{(j)}},
$$
which automatically keeps the  condition $\balpha' = \frac{\by - \bX\bbeta}{\sqrt{\lambda}}$ satisfied. Since this condition is already satisfied, we have that $y^i - \sqrt \lambda \be_{(i)}^\tr \balpha'_t - \bX^i \bbeta_t = 0$.   Thus, if we then run any row update from \eqref{eq:iz_row} we get
$$(\balpha'_{t+1}, \bbeta_{t+1}) = (\balpha'_t, \bbeta_t).$$ 
\hfill$\square$
\end{proof}

IZ's augmented projection method effectively executes RK-style updates as well as RGS updates. We can think of update \eqref{eq:iz_row} (resp.  
\eqref{eq:iz_col}) as attempting 
to satisfy the first (resp. second) condition of \eqref{eq:iz_cond} 
while maintaining the status of the other condition. 
It is the third item in each of the aforementioned claims that is surprising: each natural starting point for the IZ algorithm satisfies one of the two equations in \eqref{eq:iz_cond}; however, if one of the two equations is already satisfied at the start of the algorithm, then either the row or the column updates will have absolutely no effect through the whole algorithm.
This implies that under typical initial conditions (e.g. $\balpha' = 0, \bbeta = 0$), this approach is prone to executing \textit{many} iterations that make absolutely no progress towards convergence!  
Substituting appropriately into \eqref{RVbound}, one can get very similar convergence rates for the IZ algorithm (except that it bounds the quantity
$\|\balpha'^{T} - \balpha'^{\circ}\|^2 + \|\bbeta^T - \betaopt\|^2 $). However, a large proportion of updates do not perform any action, and we will see  in Section \ref{sec:exp} how this affects empirical convergence.

\section{Empirical Results} \label{sec:exp}

We next present simulation experiments to test the performance of 
RK, RGS and IZ (Ivanov and Zhdanov's augmented projection) algorithms in different settings of ridge regression.
For given dimensions $\rows$ and $\cols$, 
We generate a design matrix $\bX = \bU \bS \bV^\top$, 
where $\bU \in \mathbb{R}^{\rows \times k}$
, $\bV \in \mathbb{R}^{\cols \times k}$, and $k = \min(\rows,\cols)$.
Elements of $\bU$ and $\bV$ are generated from a standard Gaussian distribution and then columns are orthonormalized (to control the singular values). 
The matrix $\bS$ is a diagonal $k \times k$ matrix of singular values of $\bX$. 
The maximum singular value is 1.0 and the values decay exponentially to $\sigma_{\min}$. The true  parameter vector $\bbeta$ is generated from a multivariate Gaussian distribution with zero mean and identity covariance. The vector $\by$ is generated by adding independent standard Gaussian noise to the coordinates of $\bX \bbeta$.
We use different values of $\rows$, $\cols$, $\lambda$ and $\sigma_{min}$ as listed in Table \ref{tbl:params}. For each configuration of the simulated parameters,
we run RGS and RK and IZ for $10^4$ iterations on a random instance of that configuration and report the Euclidean difference between estimated and optimal parameters  after each 100 iterations. We average the error over 20 regression problems generated from the same configuration.
We used several different initializations for the IZ algorithm as shown in 
Table \ref{tbl:algos}. Experiments where implemented in MATLAB and executed 
on an Intel Core i7-6700K 4GHz machine with 16GB of RAM.

The results are reported in Figures \ref{fig:nep}, 
\ref{fig:ngp} and \ref{fig:nlp}. Figure \ref{fig:nep} shows that 
RGS and RK exhibit similar behavior when $\rows = \cols$. 
Poor conditioning of the design matrix results in slower convergence. However, the effect of conditioning is most apparent when the regularization parameter is small. 
Figures \ref{fig:ngp} and \ref{fig:nlp} show that RGS  
consistently outperforms other methods when $\rows > \cols$ while RK consistently outperforms other methods when $\rows < \cols$. The difference is again most apparent when the regularization parameter is small (note that even when $\lambda=0$, RK is known to converge to the least-norm solution, see e.g. \cite{MaConvergence15}). 
The notably poor performance of RK and RGS in the $\rows > \cols$
and $\cols > \rows$ cases respectively when $\lambda$ is very small 
agrees with Theorem \ref{thm:rates}:
a small $\lambda$ results in an expected error reduction ratio that is very close to 1.

Another error metric that is commonly used in evaluating numerical methods
is $\|X^*X \bbeta_t - X^* y\|$. Intuitively, this metric measures how well the value of 
$\bbeta_t$ fits the system of equations. We examined how the tested methods behave in terms 
of this metric when $\lambda = 10^{-3}$. The results are depicted in Figure \ref{fig:xnrm}.
We see that RGS is achieving very small error in the $\rows < \cols$ case despite its 
poor performance in terms of the solution recovery error $\|\beta_t - \bbeta^o\|$.
This is expected since when $\lambda$ tends to 0, RGS tends to be solving an underdetermined system of equations where the are infinite solutions that fit. 
This is not  the case for the overdetermined case $\rows > \cols$ and therefore we see that RK method is performing poorly in terms of $\|X^*X \bbeta_t - X^* y\|$ 
in that case.

Looking at IZ methods, we notice that IZ0 (resp. IZ1) exhibit similar convergence behavior as that of RK (resp. RGS) although typically slower. This agrees with our analysis which reveals that, depending on the initialization,
IZ can perform RGS or RK-style updates except that some iterations can be ineffective, which causes slower convergence.
Interestingly, IZMIX, where $\alpha$ is initialized midway between IZ0 and IZ1 exhibits convergence behavior that in most cases is in between IZ0 and IZ1.

\begin{table}[ht]
\centering
\begin{tabular}{|c|l|c|}
\hline
Parameter & \multicolumn{1}{|c|}{Definition}   & Values \\
\hline
$(\rows,\cols)$ & \begin{tabular}{l} Dimensions of the design matrix $\bX$ \end{tabular} 
& $(1000, 1000)$, $(10^4, 100)$ , $(100, 10^4)$ \\
\hline
$\lambda$ & \begin{tabular}{l} Regularization parameter \end{tabular}  & $10^{-3},10^{-2},10^{-1}$\\
\hline
$\sigma_{min}$ & 
\begin{tabular}{l}
Minimum singular value of the design matrix \\ ($\sigma_{max} = 1.0$) 
\end{tabular} & $1.0,10^{-1},10^{-2},10^{-3}$\\
\hline
\end{tabular}
\caption{Different parameters used in simulation experiments}
\label{tbl:params}
\end{table}

\begin{table}[ht]
\centering
\begin{tabular}{|c|p{10cm}|}
\hline
Algorithm & Description \\
\hline
RGS & Randomized Gauss-Siedel updates using \eqref{eq:RGSridge1} with initialization $\beta_0 = 0$\\
\hline
RK & Randomized Kaczmarz updates using \eqref{eq:kerkacz3} with initialization
 $\alpha_0 = 0$ \\
\hline
IZ0 & Ivanov and Zhdanov's augmented projection method with $\balpha_0 = 0, \bbeta_0 = 0$ \\
\hline
IZ1 & Ivanov and Zhdanov's augmented projection method with $\balpha_0 = \by / \sqrt{\lambda}, \bbeta_0 = 0$ \\
\hline
IZMIX & Ivanov and Zhdanov's augmented projection method with $\balpha_0 = \by / 2\sqrt{\lambda}, \bbeta_0 = 0$ \\
\hline
IZRND & Ivanov and Zhdanov's augmented projection method with elements of $\bbeta_0$ and $\balpha_0$ randomly drawn from a standard normal distribution \\
\hline
\end{tabular}
\caption{List of algorithms compared in simulation experiment.}
\label{tbl:algos}
\end{table}

\begin{figure}[p]
\centering
\includegraphics[scale=0.3]{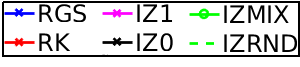} 
\begin{tabular}{m{1cm} >{\centering\arraybackslash}m{5cm} >{\centering\arraybackslash}m{5cm} >{\centering\arraybackslash}m{5cm}}
$\sigma_{min}$ & $\lambda=10^{-3}$ & $\lambda=10^{-2}$ & $\lambda=10^{-1}$ \\
1.0 &  \includegraphics[scale=0.35]{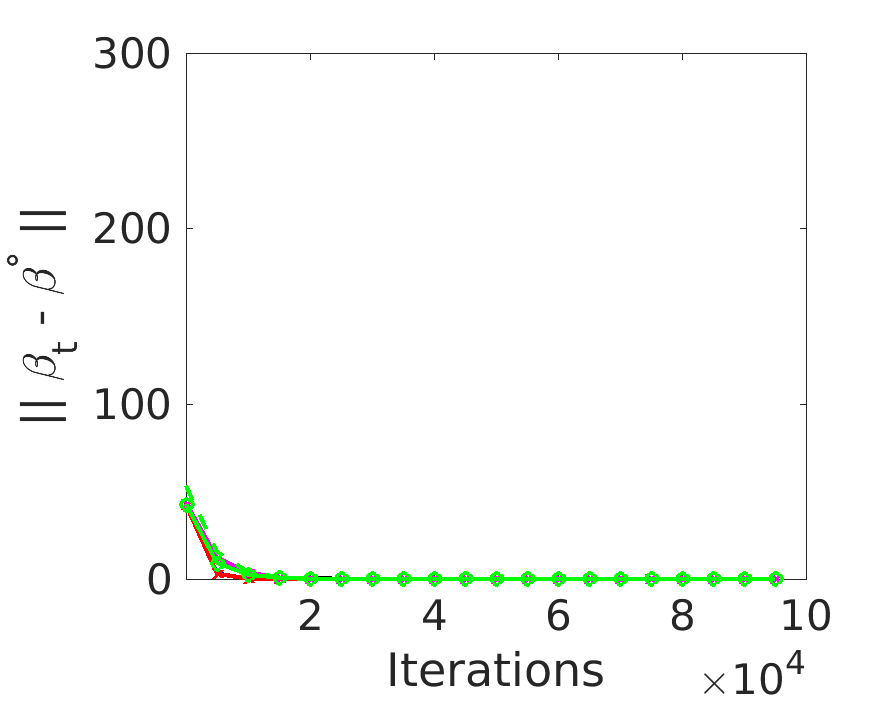}
& \includegraphics[scale=0.35]{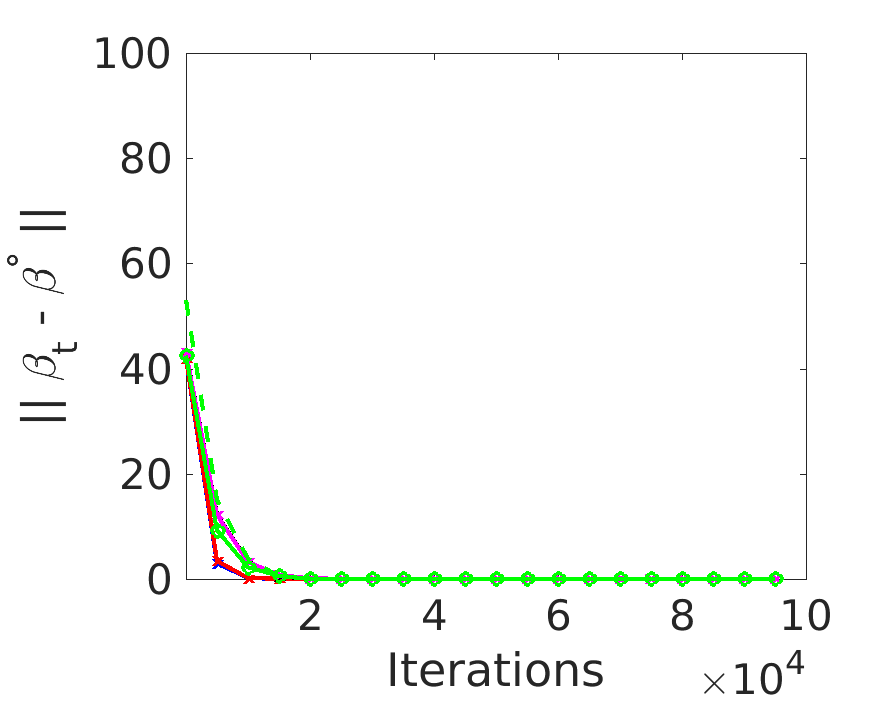}
& \includegraphics[scale=0.35]{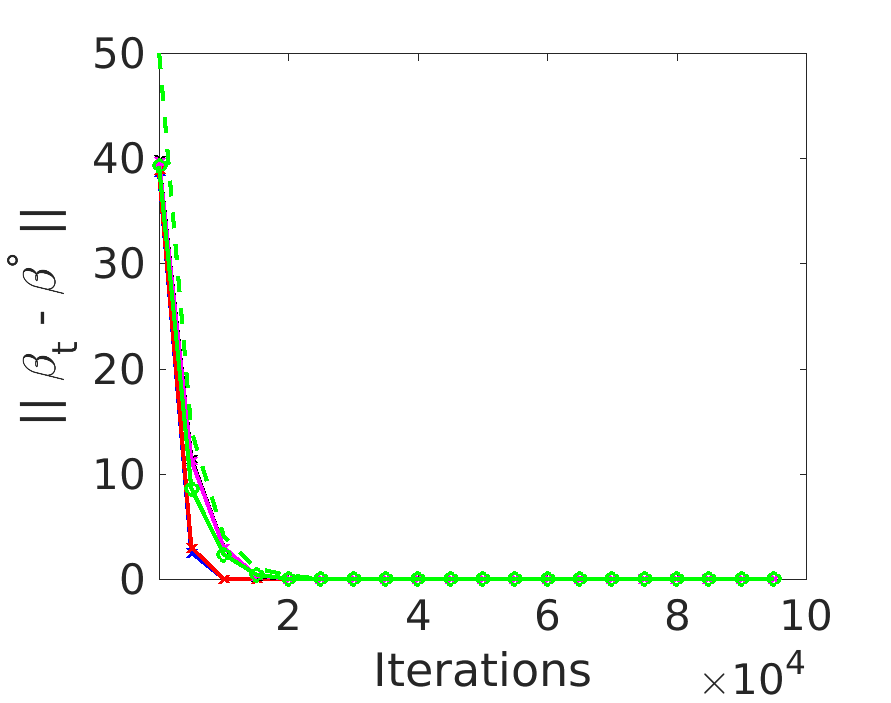} \\
$10^{-1}$ &  \includegraphics[scale=0.35]{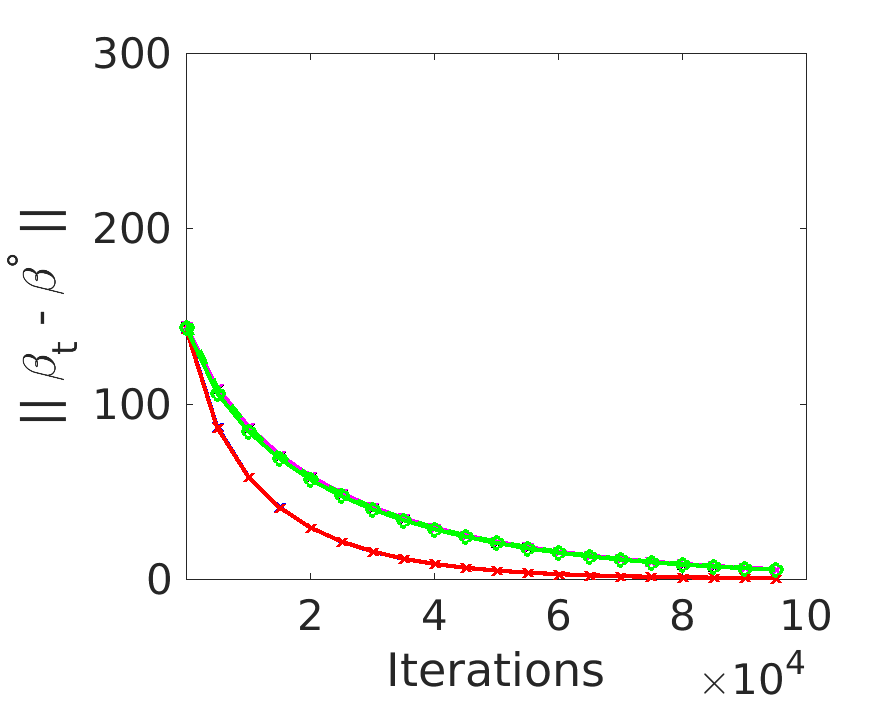}
& \includegraphics[scale=0.35]{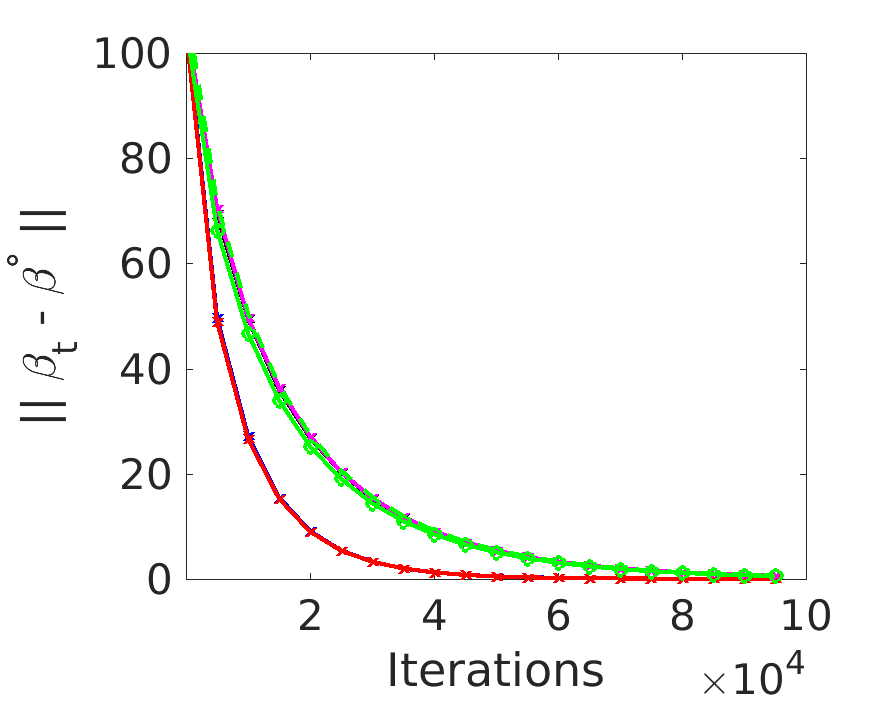}
& \includegraphics[scale=0.35]{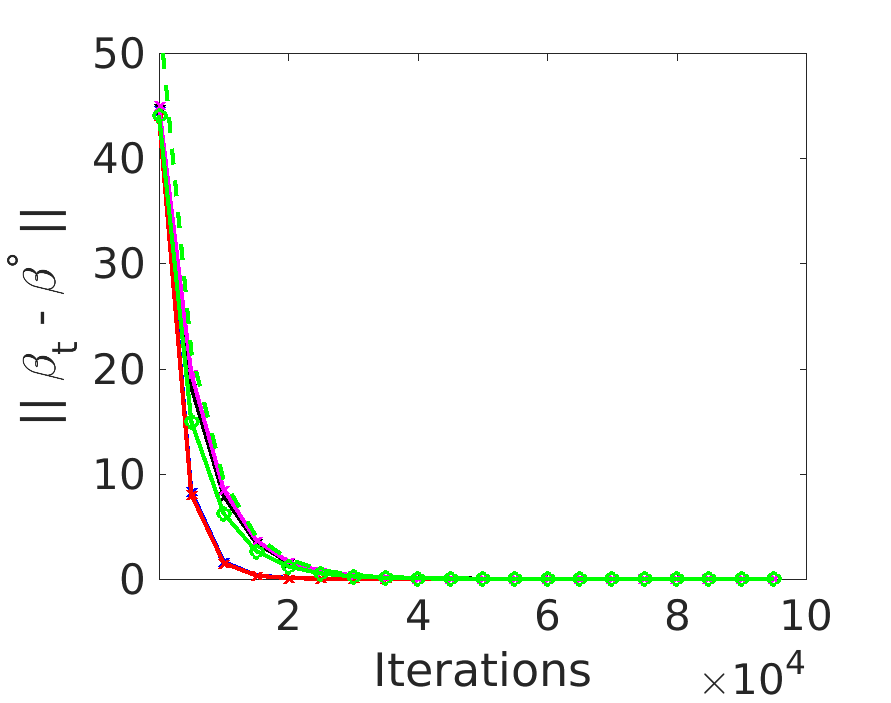} \\
$10^{-2}$ &  \includegraphics[scale=0.35]{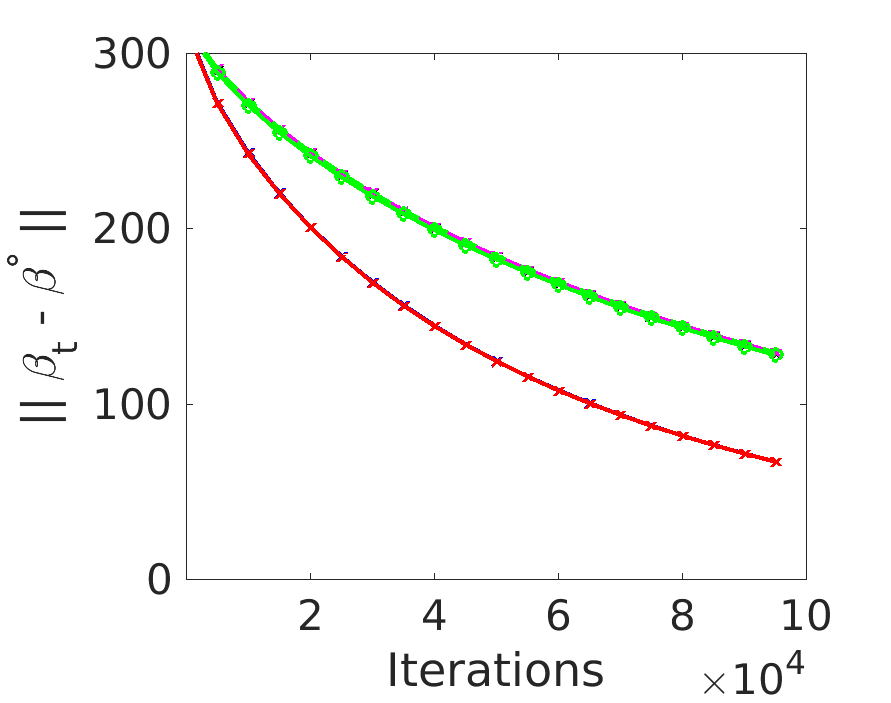}
& \includegraphics[scale=0.35]{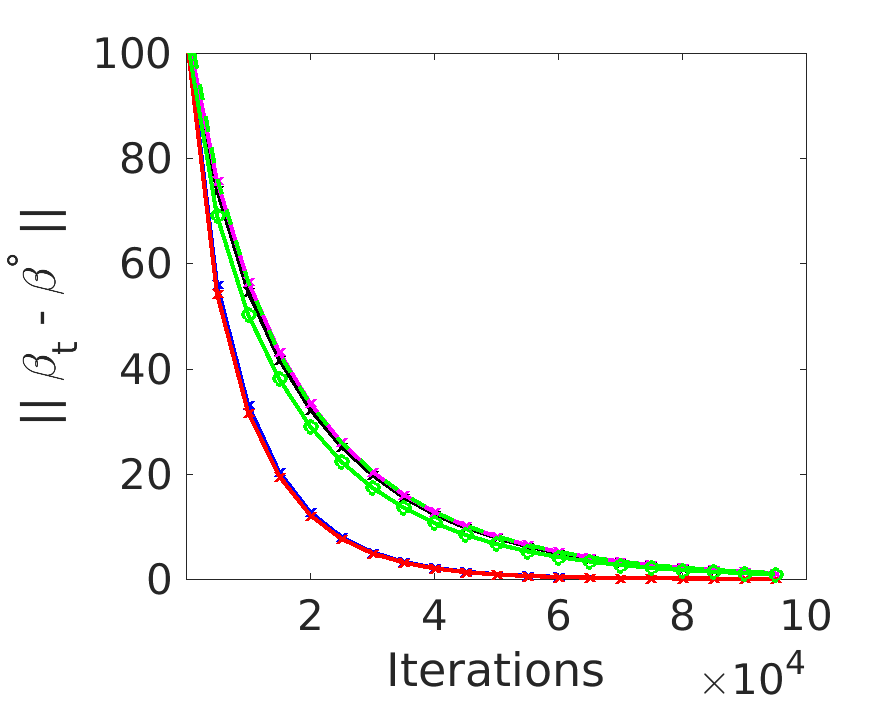}
& \includegraphics[scale=0.35]{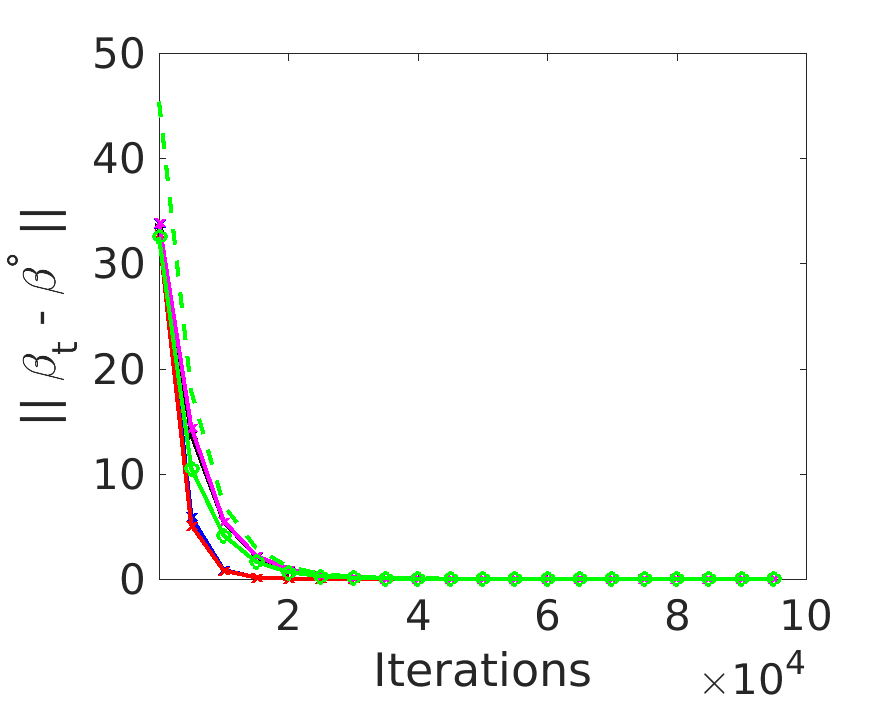} \\
$10^{-3}$ &  \includegraphics[scale=0.35]{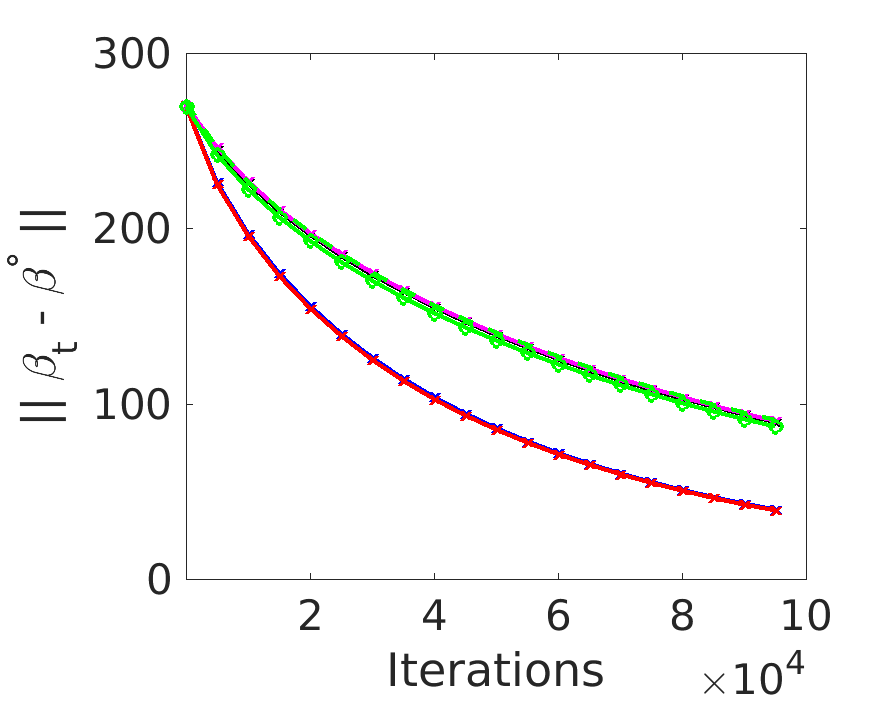}
& \includegraphics[scale=0.35]{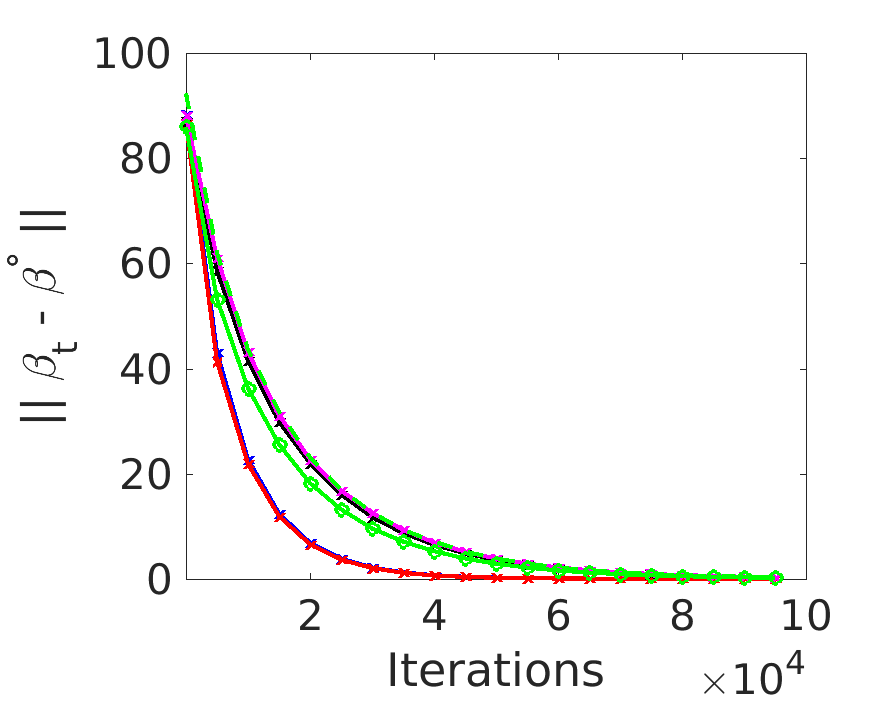}
& \includegraphics[scale=0.35]{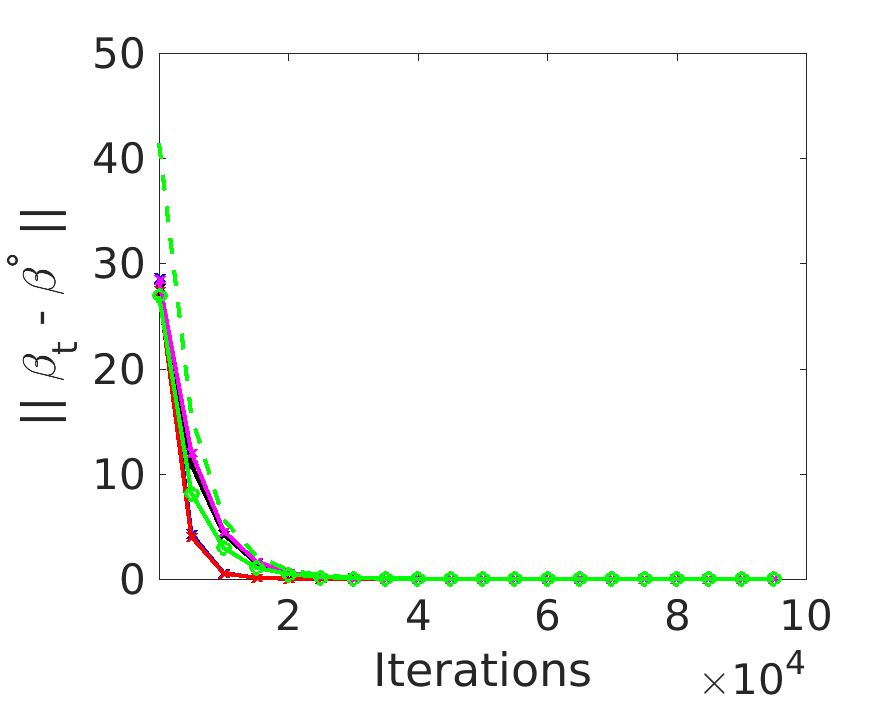} \\
\end{tabular}
\caption{Simulation results for $\rows = \cols = 1000$: Euclidean error $\|\bbeta_t - \betaopt\|$ versus iteration count.}
\label{fig:nep}
\end{figure}
\begin{figure}[p]
\centering
\includegraphics[scale=0.3]{legend.png} 
\begin{tabular}{m{1cm} >{\centering\arraybackslash}m{5cm} >{\centering\arraybackslash}m{5cm} >{\centering\arraybackslash}m{5cm}}
$\sigma_{min}$ & $\lambda=10^{-3}$ & $\lambda=10^{-2}$ & $\lambda=10^{-1}$ \\
1.0 &  \includegraphics[scale=0.35]{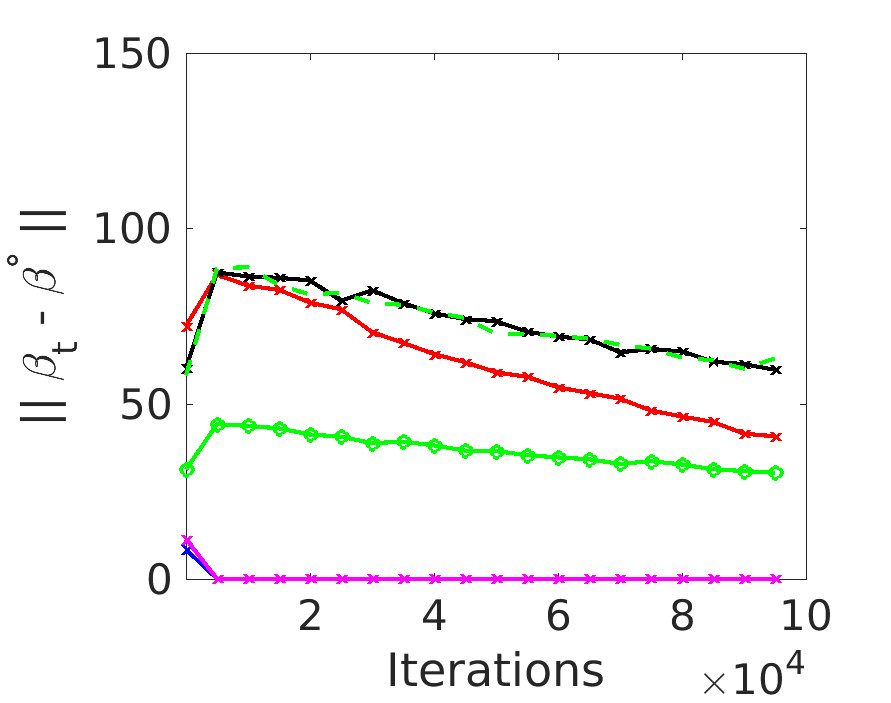}
& \includegraphics[scale=0.35]{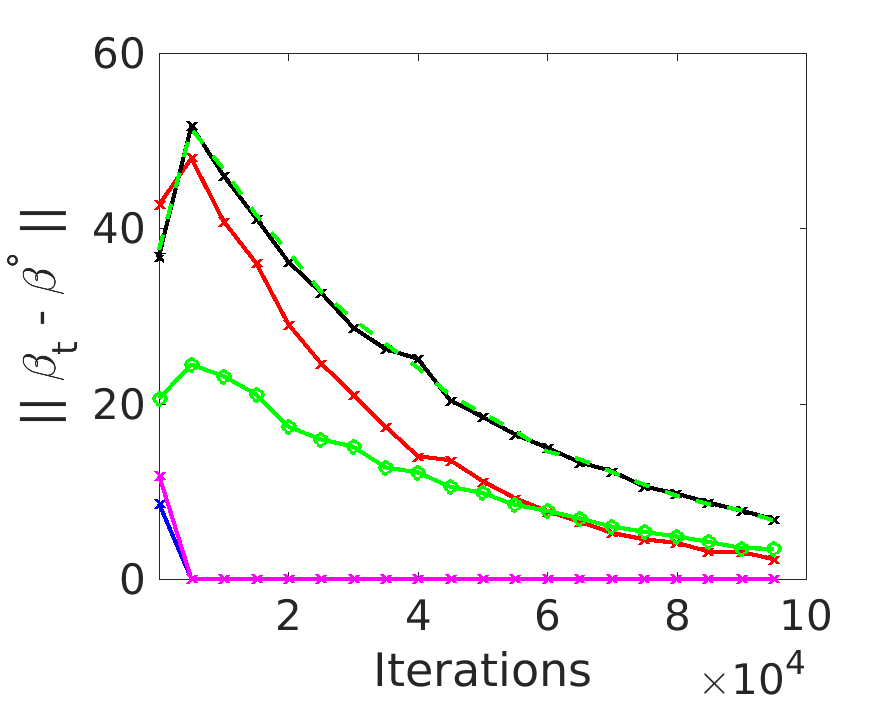}
& \includegraphics[scale=0.35]{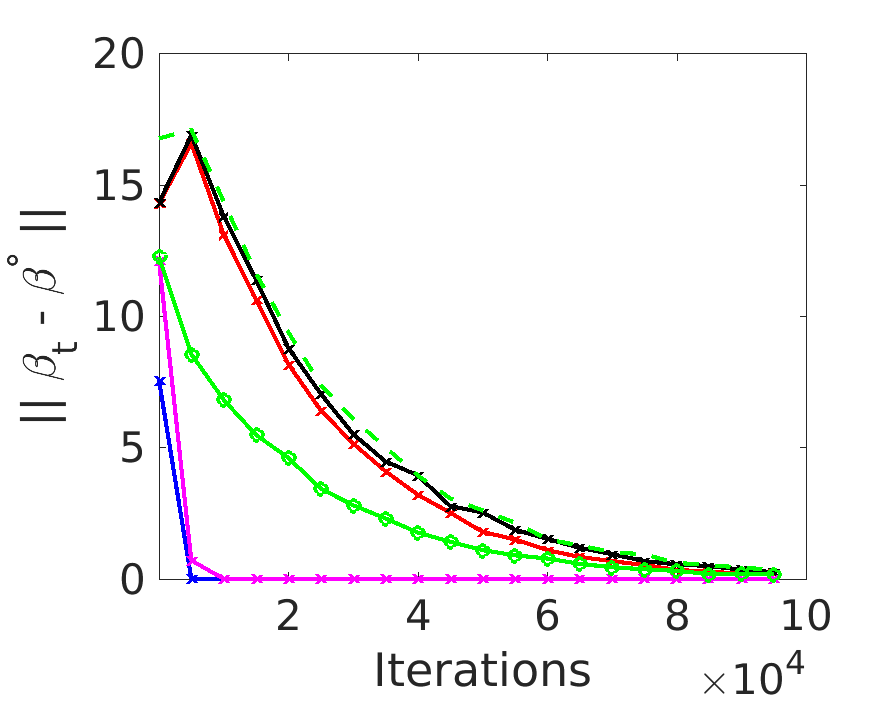} \\
$10^{-1}$ &  \includegraphics[scale=0.35]{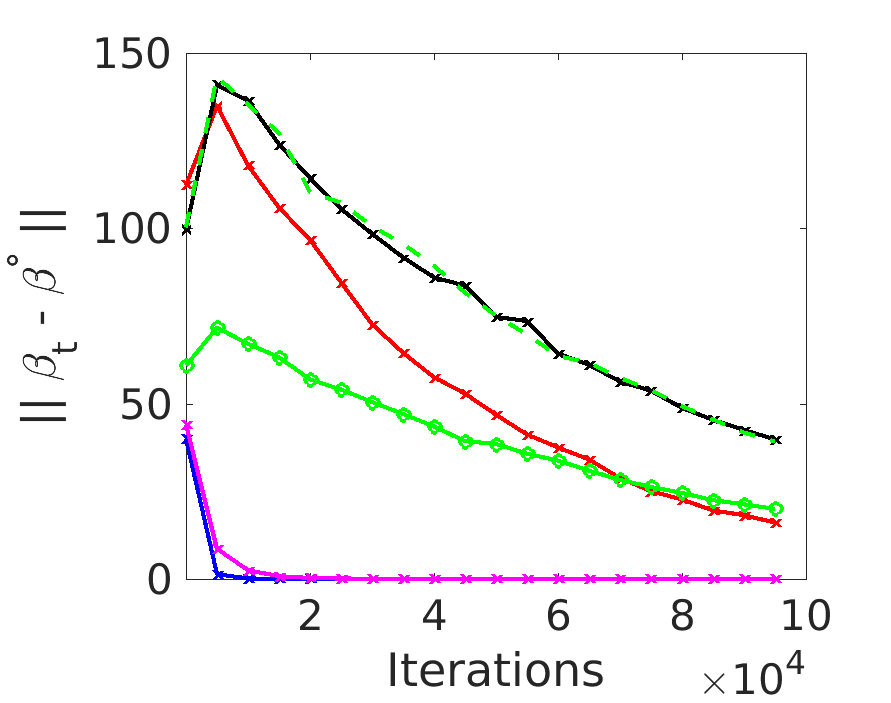}
& \includegraphics[scale=0.35]{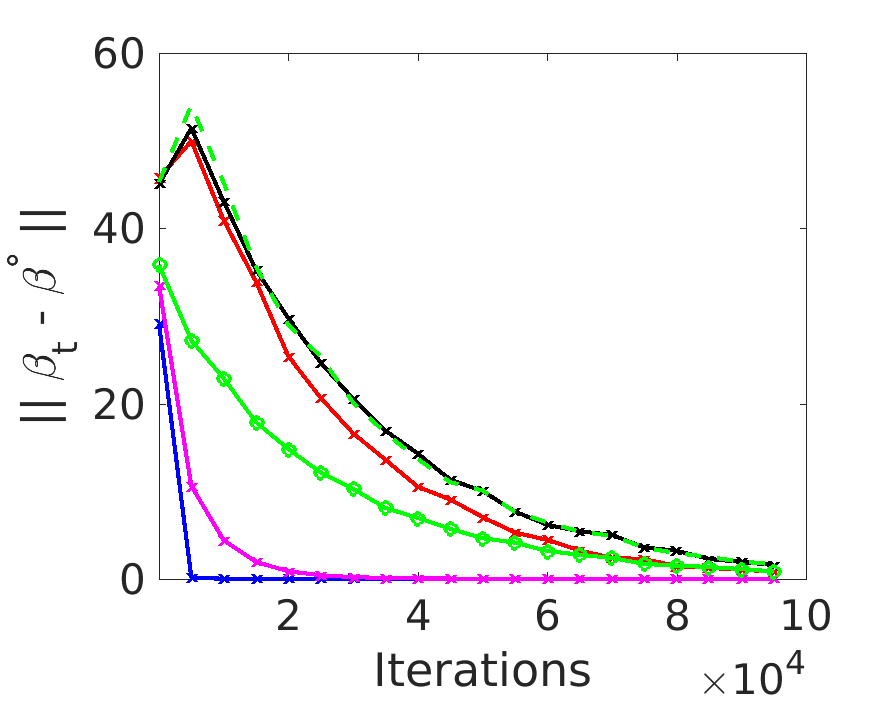}
& \includegraphics[scale=0.35]{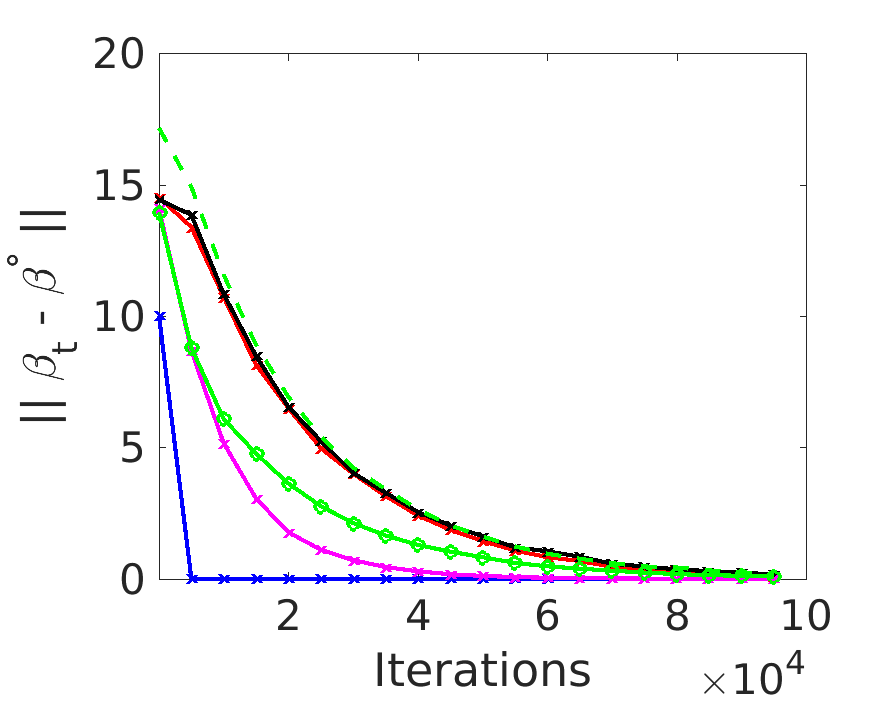} \\
$10^{-2}$ &  \includegraphics[scale=0.35]{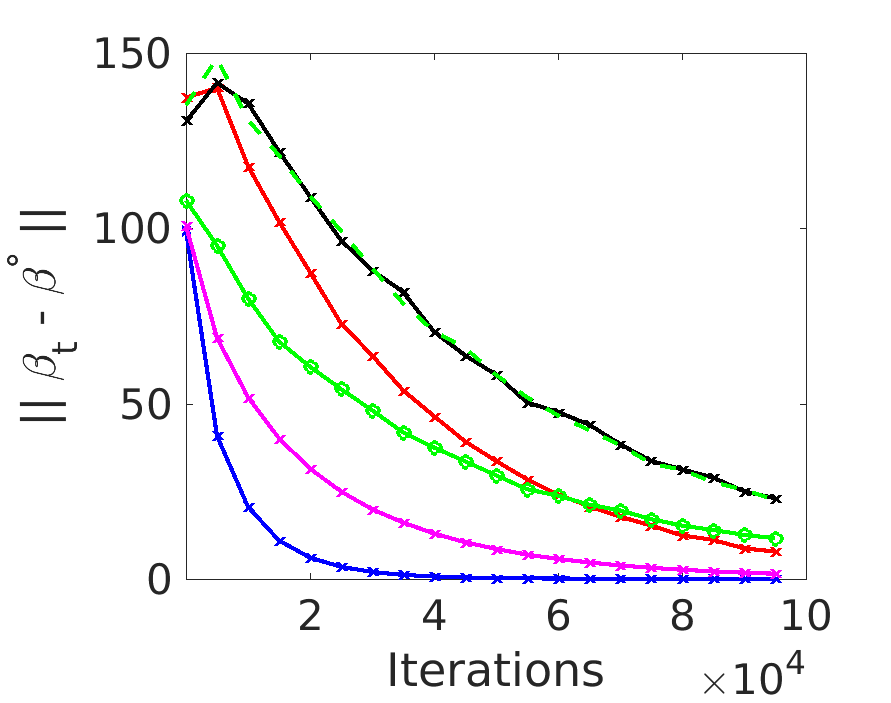}
& \includegraphics[scale=0.35]{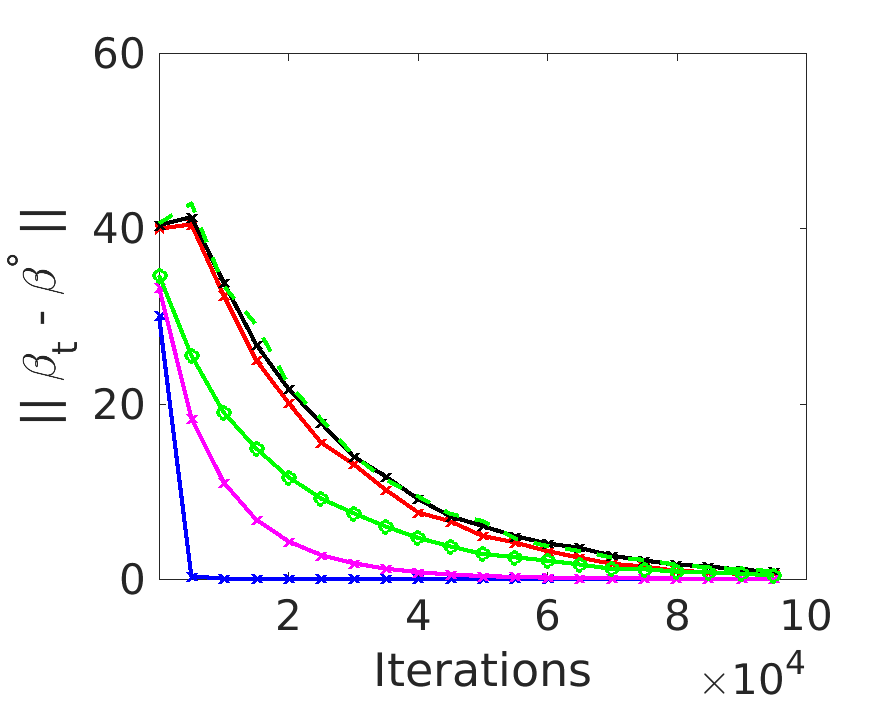}
& \includegraphics[scale=0.35]{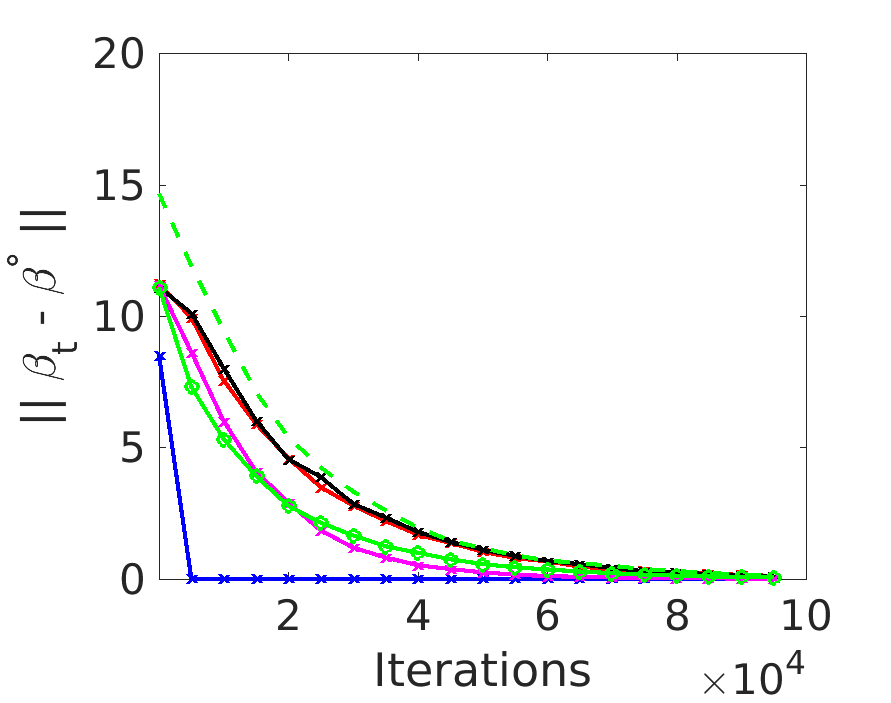} \\
$10^{-3}$ &  \includegraphics[scale=0.35]{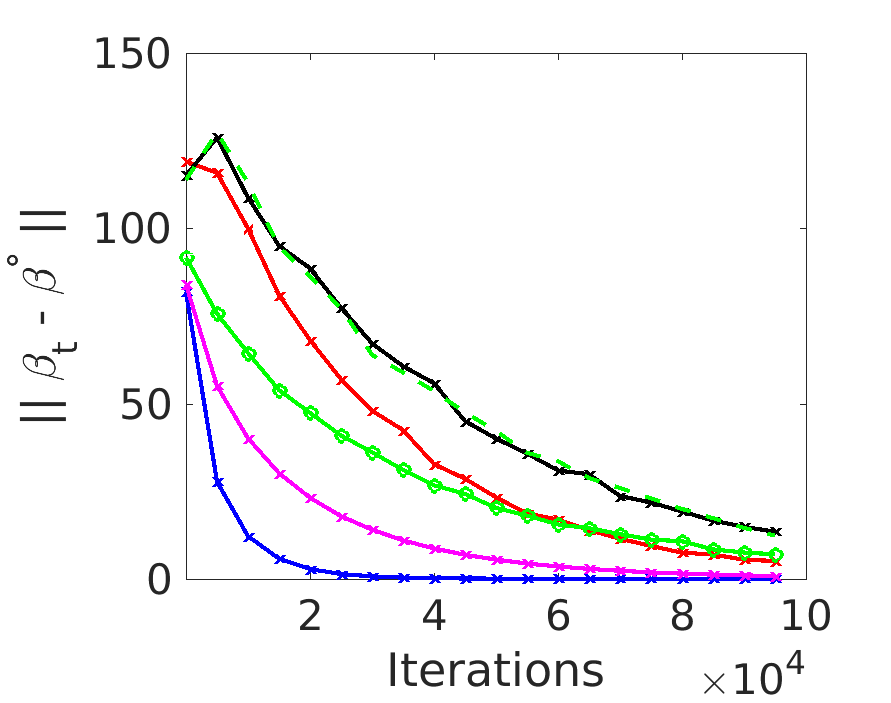}
& \includegraphics[scale=0.35]{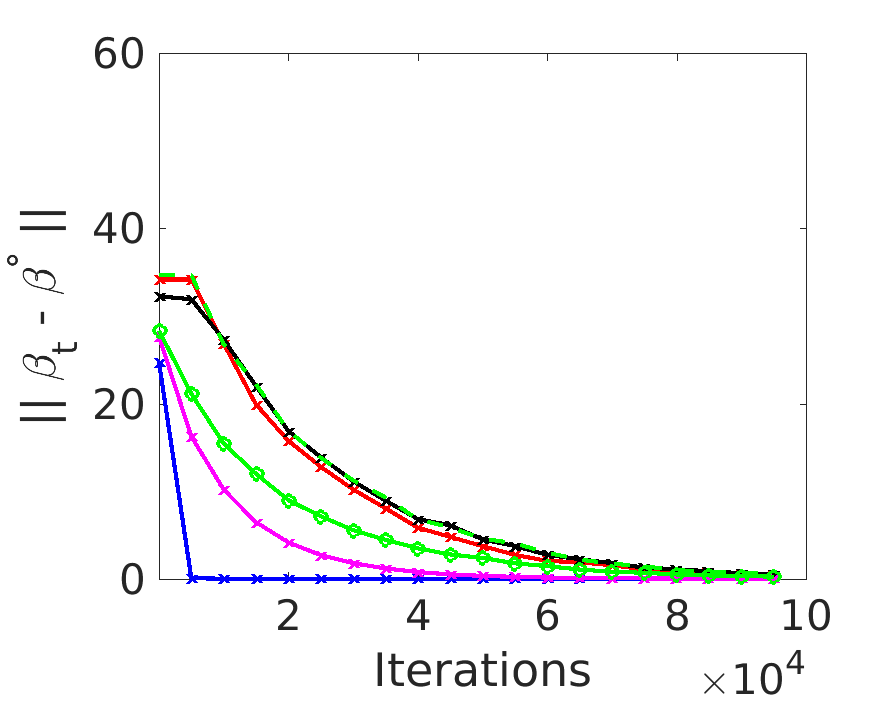}
& \includegraphics[scale=0.35]{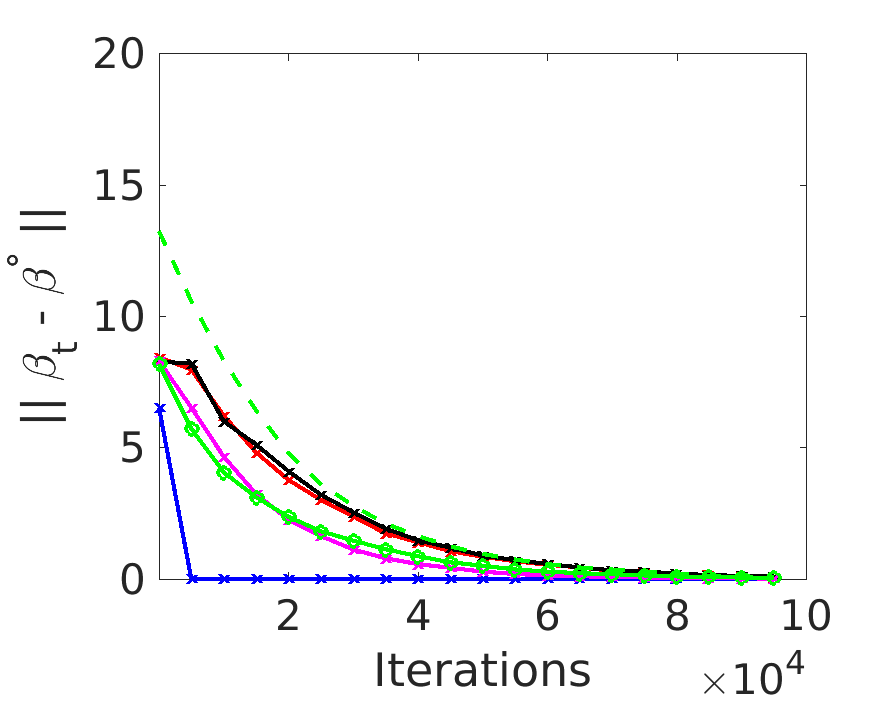} \\
\end{tabular}
\caption{Simulation results for $\rows = 10^4, \cols = 100$: Euclidean error $\|\bbeta_t - \betaopt\|$ versus iteration count.}
\label{fig:ngp}
\end{figure}
\begin{figure}[p]
\centering
\includegraphics[scale=0.3]{legend.png} 
\begin{tabular}{m{1cm} >{\centering\arraybackslash}m{5cm} >{\centering\arraybackslash}m{5cm} >{\centering\arraybackslash}m{5cm}}
$\sigma_{min}$ & $\lambda=10^{-3}$ & $\lambda=10^{-2}$ & $\lambda=10^{-1}$ \\
1.0 &  \includegraphics[scale=0.35]{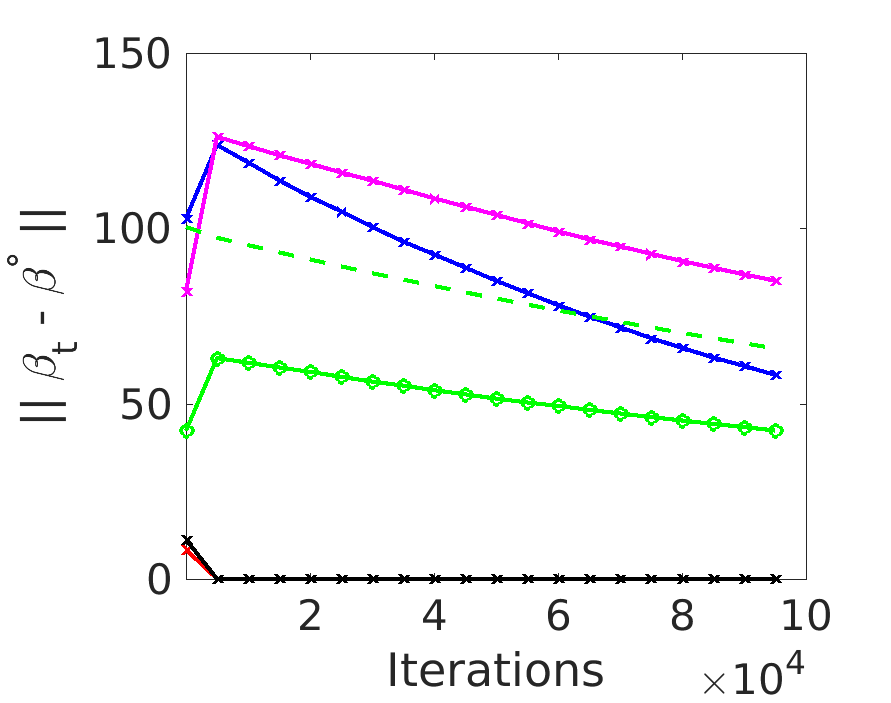}
& \includegraphics[scale=0.35]{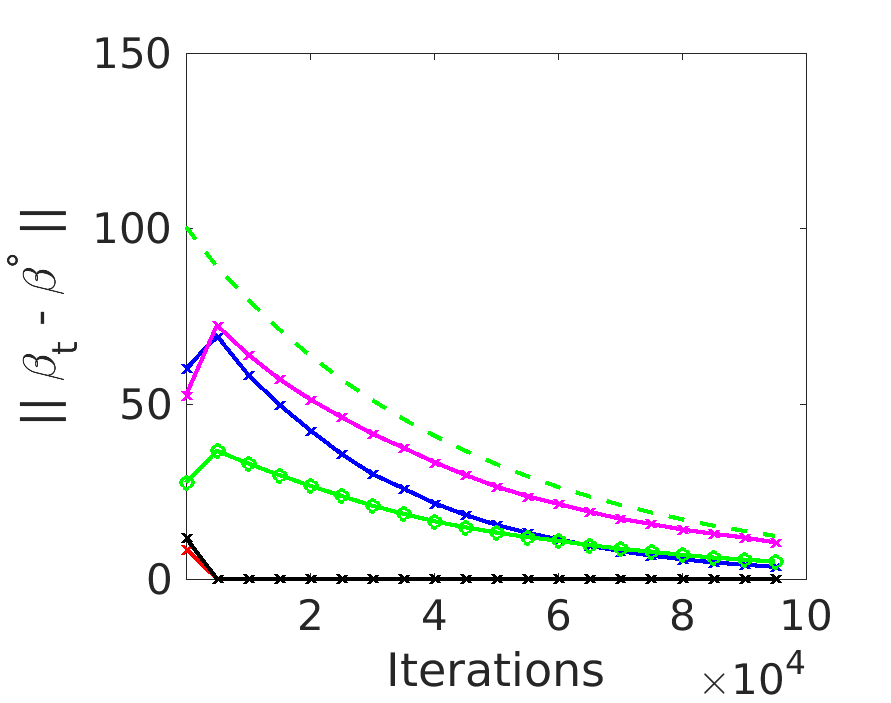}
& \includegraphics[scale=0.35]{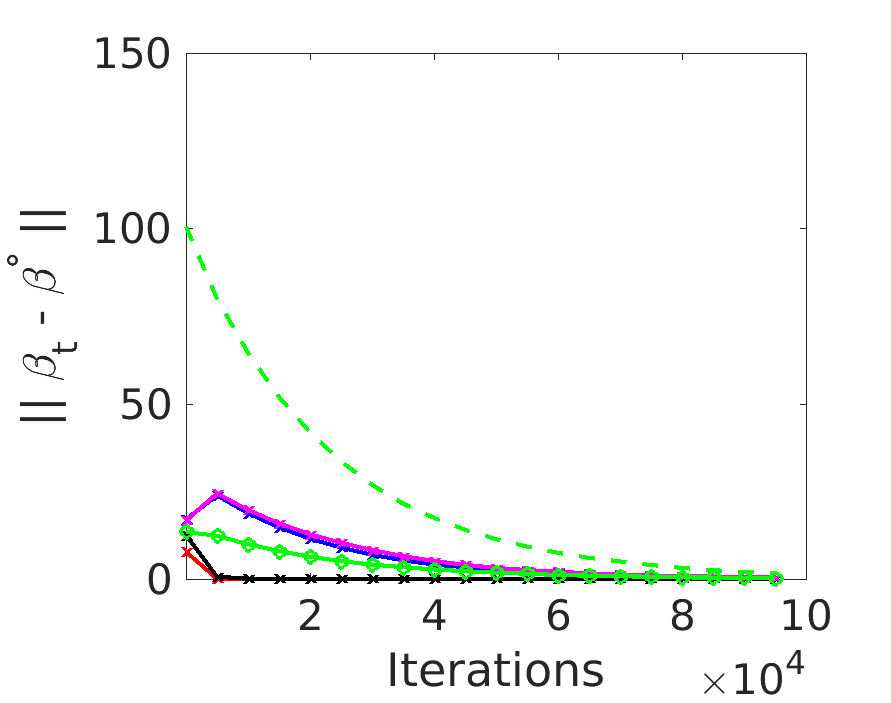} \\
$10^{-1}$ &  \includegraphics[scale=0.35]{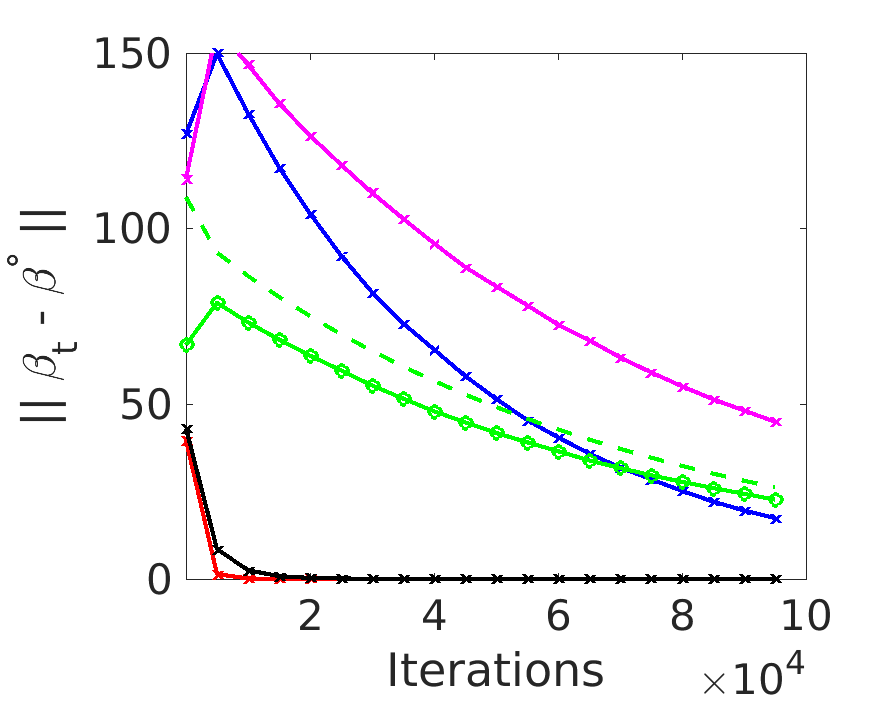}
& \includegraphics[scale=0.35]{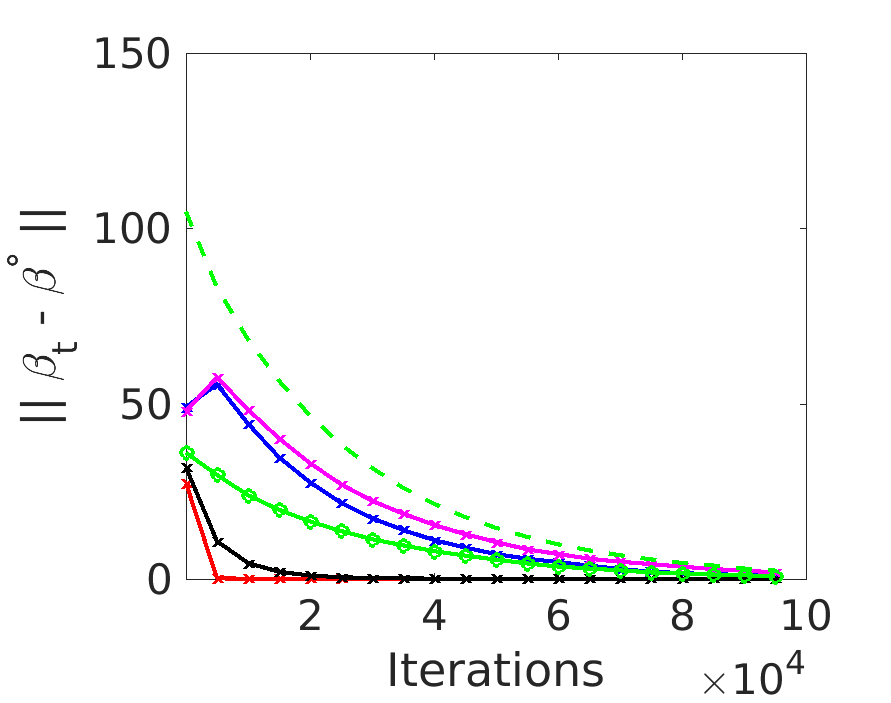}
& \includegraphics[scale=0.35]{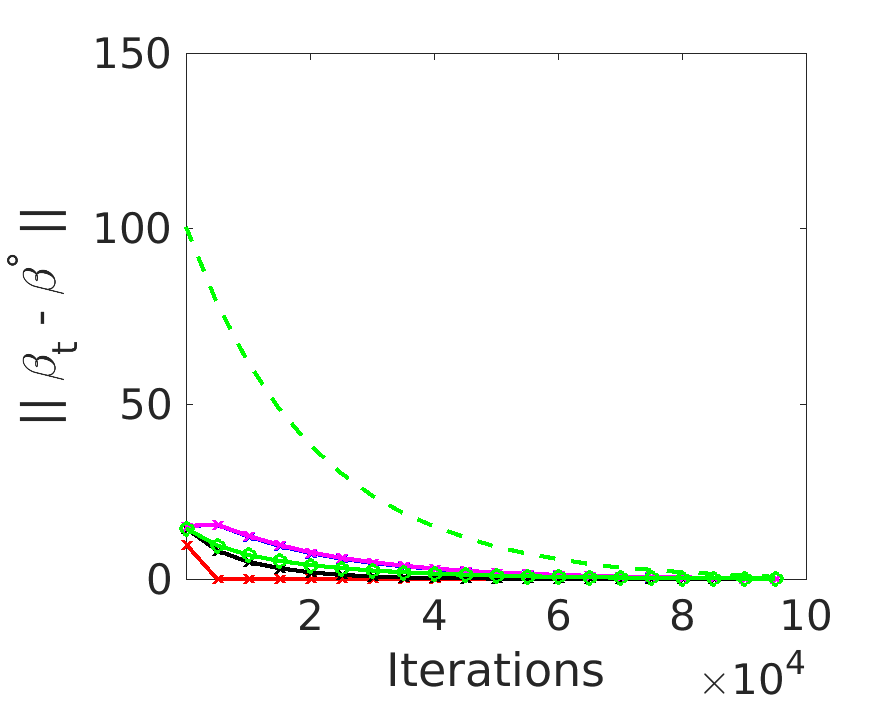} \\
$10^{-2}$ &  \includegraphics[scale=0.35]{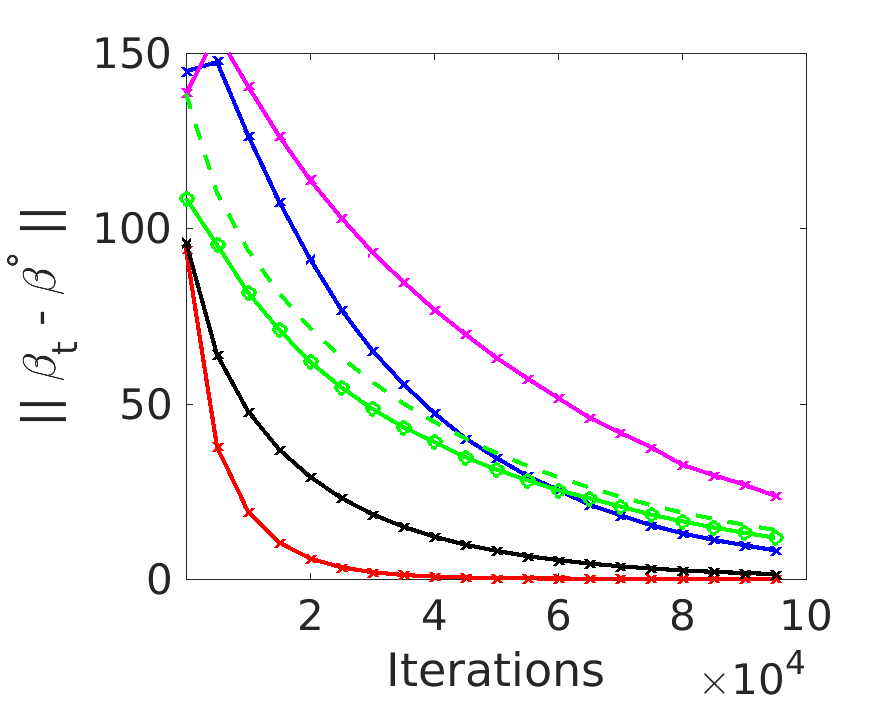}
& \includegraphics[scale=0.35]{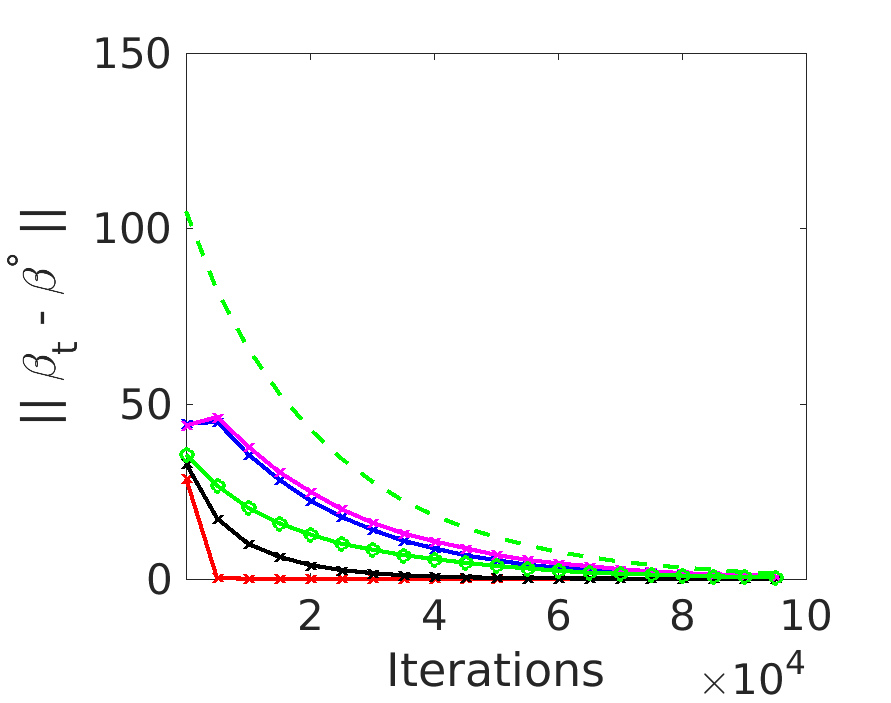}
& \includegraphics[scale=0.35]{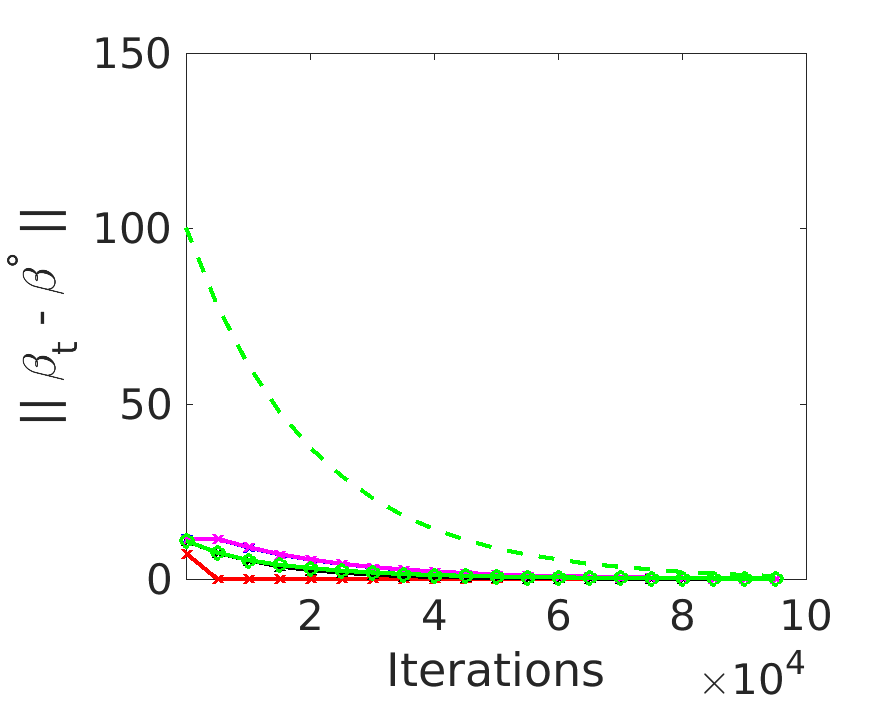} \\
$10^{-3}$ &  \includegraphics[scale=0.35]{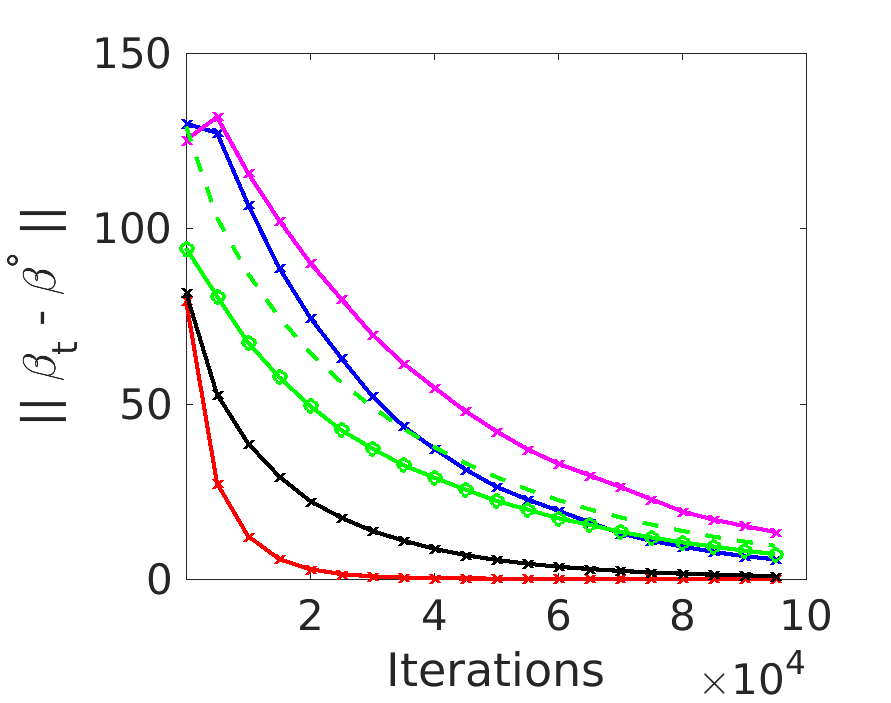}
& \includegraphics[scale=0.35]{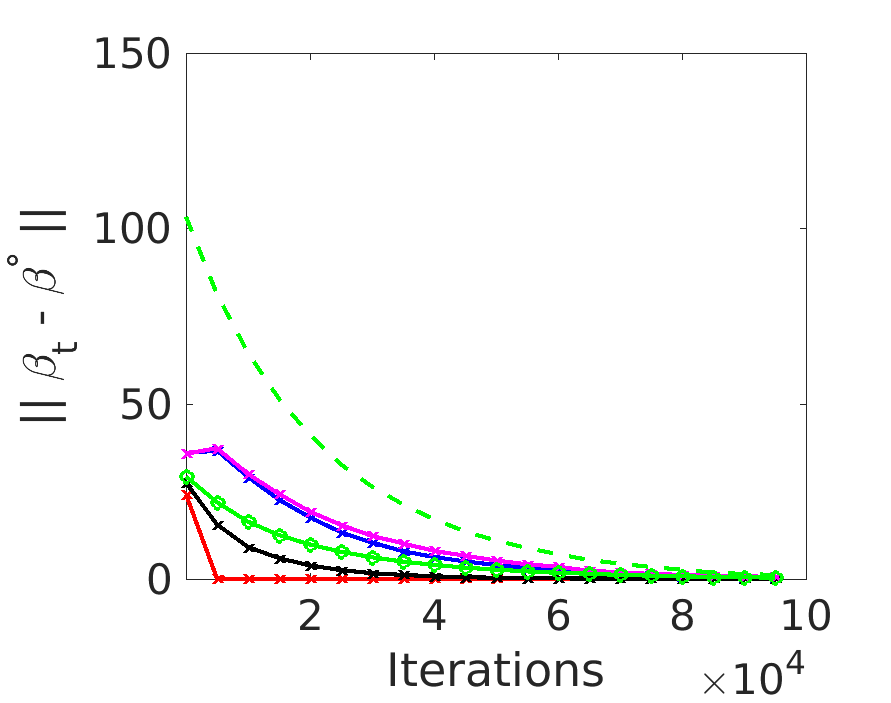}
& \includegraphics[scale=0.35]{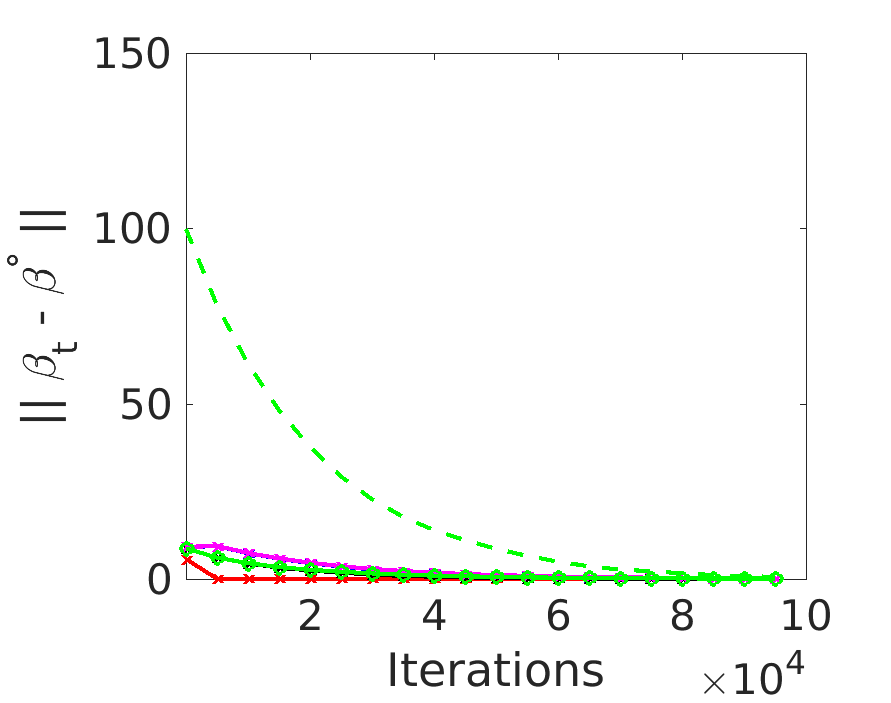} \\
\end{tabular}
\caption{Simulation results for $\rows = 100, \cols = 10^4$: Euclidean error $\|\bbeta_t - \betaopt\|$ versus iteration count.}
\label{fig:nlp}
\end{figure}
\begin{figure}[p]
\centering
\includegraphics[scale=0.3]{legend.png} 
\begin{tabular}{m{1cm} >{\centering\arraybackslash}m{5cm} >{\centering\arraybackslash}m{5cm} >{\centering\arraybackslash}m{5cm}}
$\sigma_{min}$ & $m=n=1000$ & $10000=m>n=100$ & $100=m<n=10000$ \\
1.0 &  \includegraphics[scale=0.35]{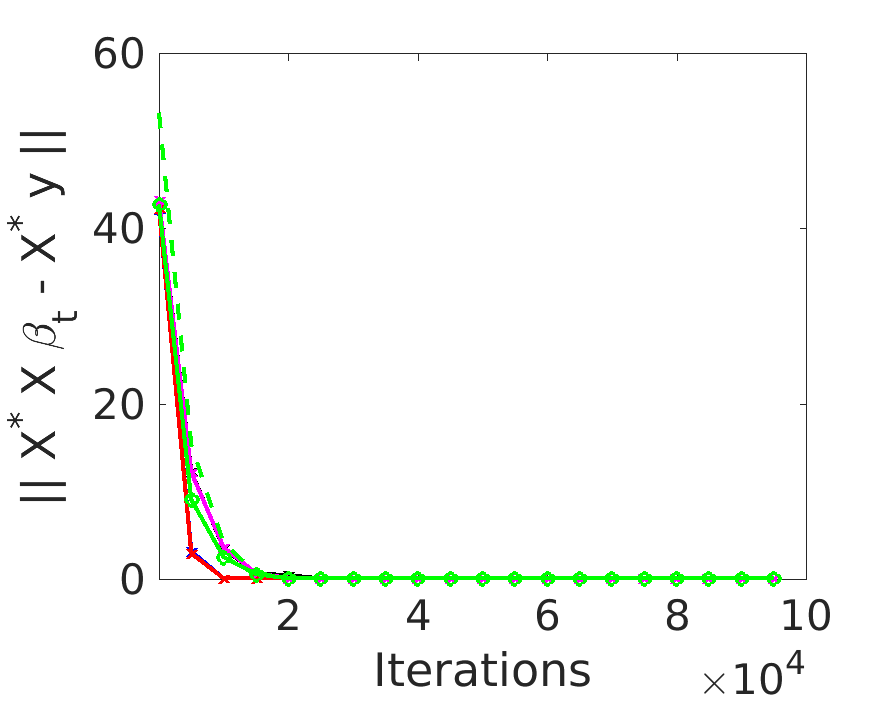}
& \includegraphics[scale=0.35]{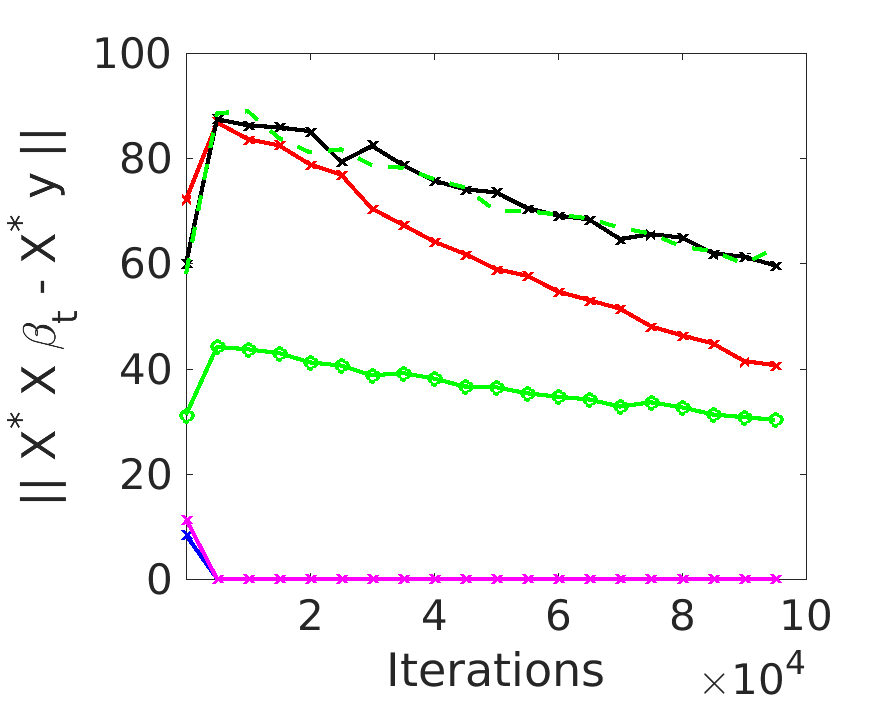}
& \includegraphics[scale=0.35]{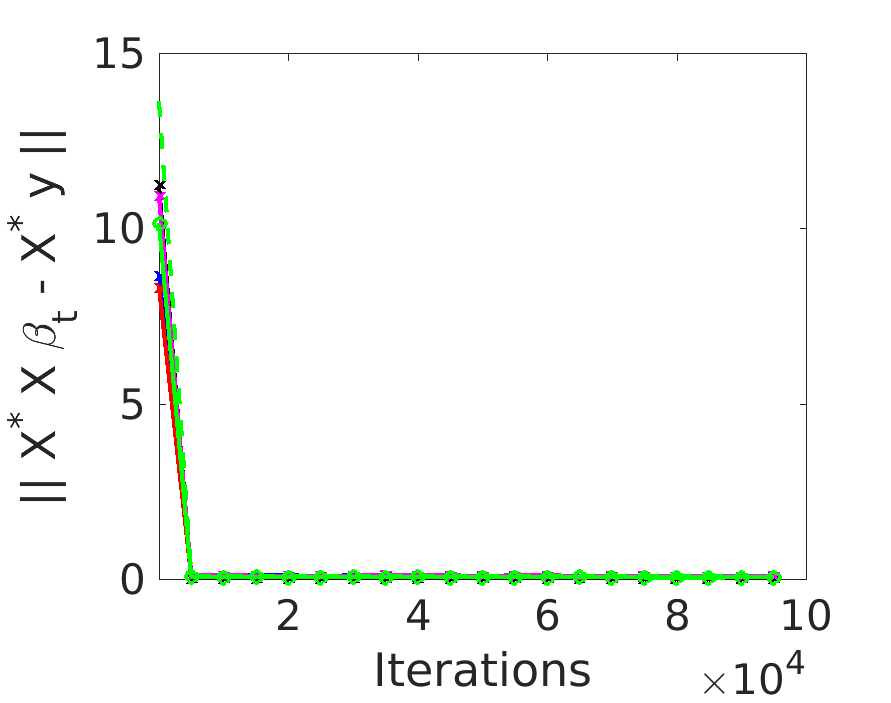} \\
$10^{-1}$ &  \includegraphics[scale=0.35]{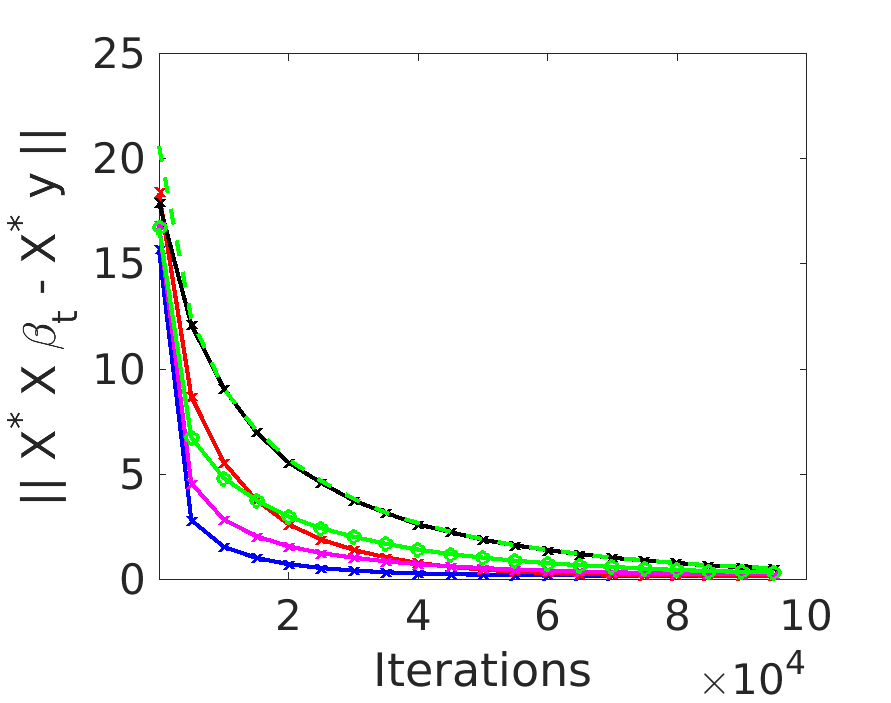}
& \includegraphics[scale=0.35]{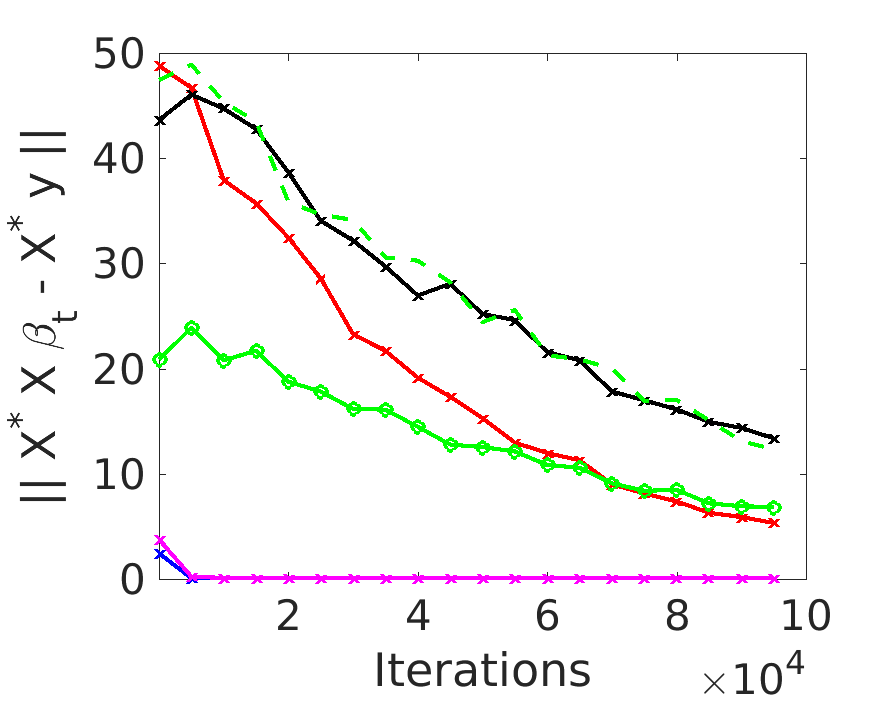}
& \includegraphics[scale=0.35]{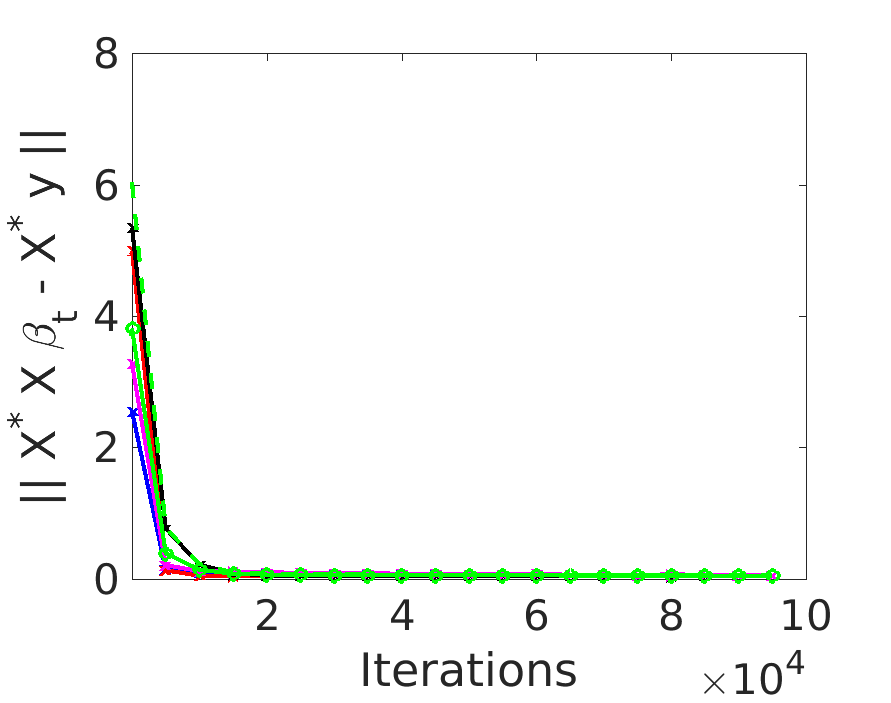} \\
$10^{-2}$ &  \includegraphics[scale=0.35]{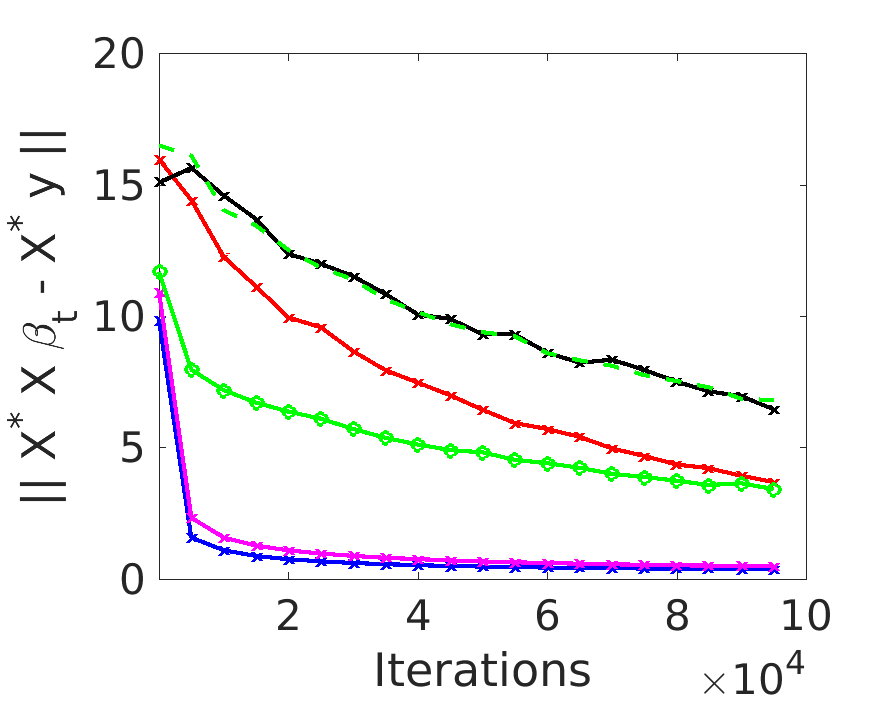}
& \includegraphics[scale=0.35]{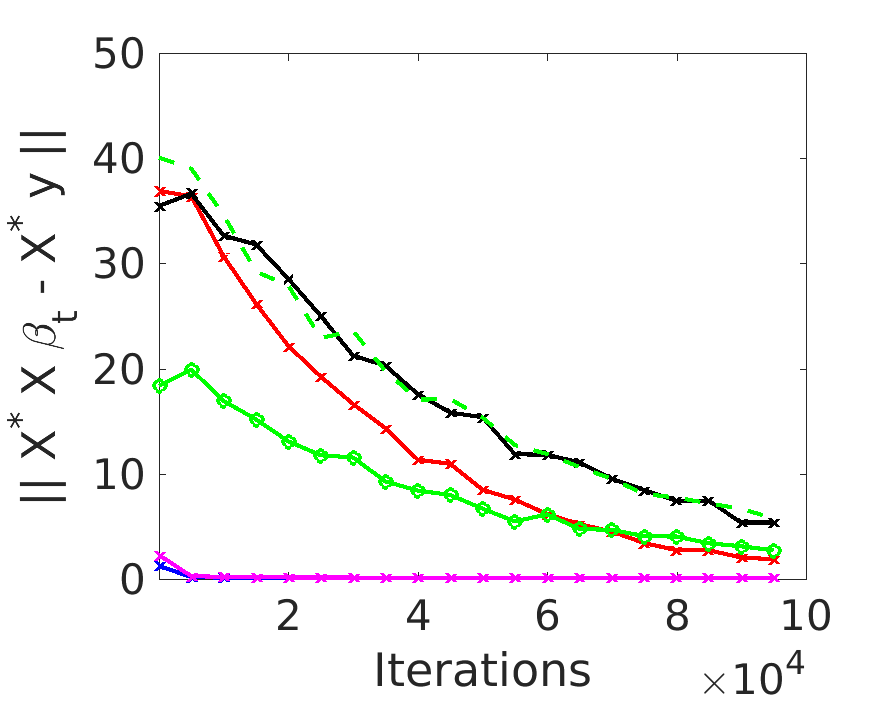}
& \includegraphics[scale=0.35]{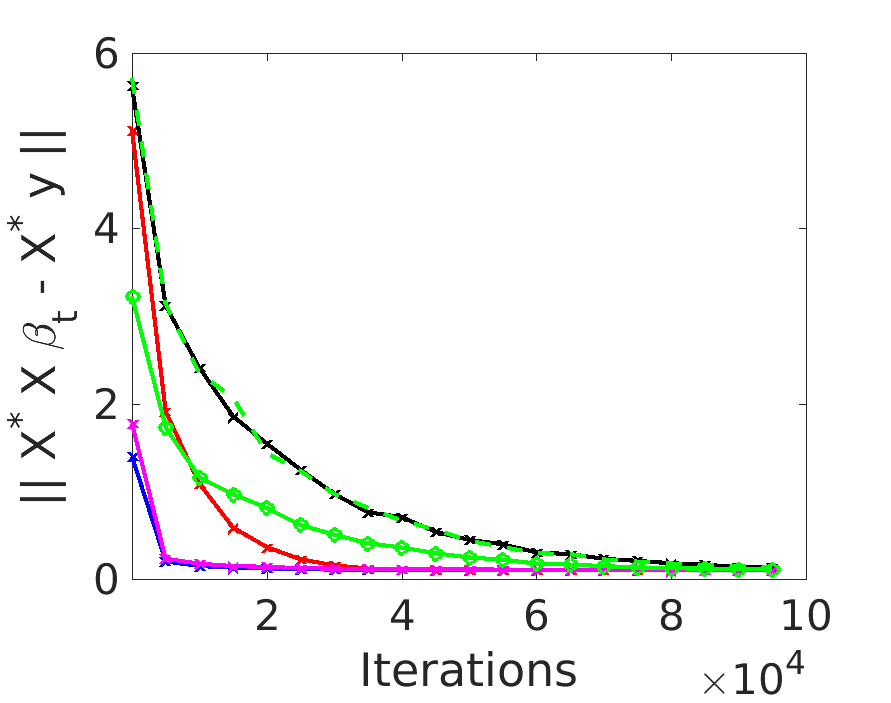} \\
$10^{-3}$ &  \includegraphics[scale=0.35]{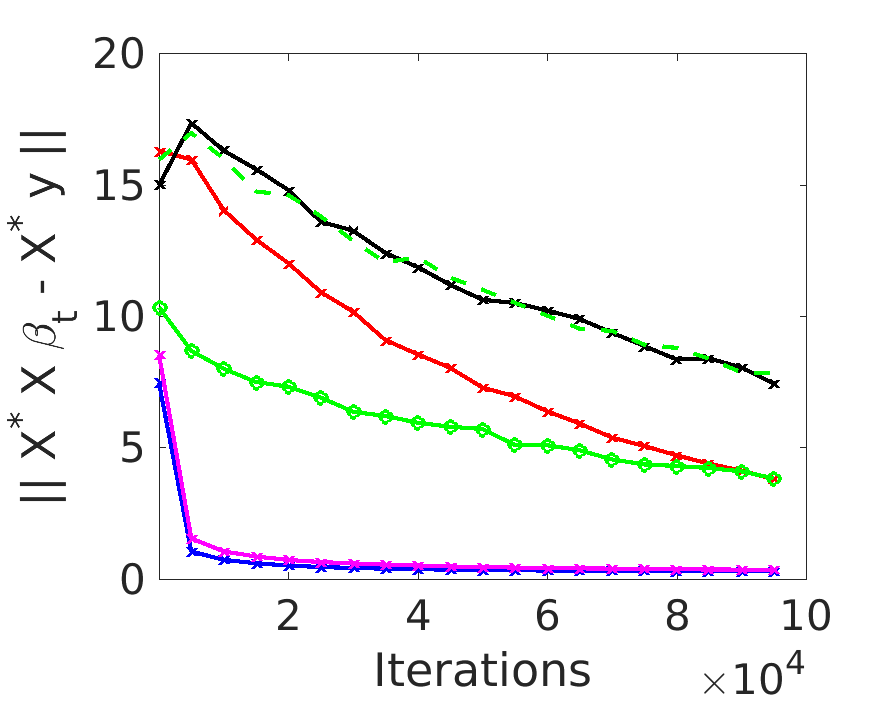}
& \includegraphics[scale=0.35]{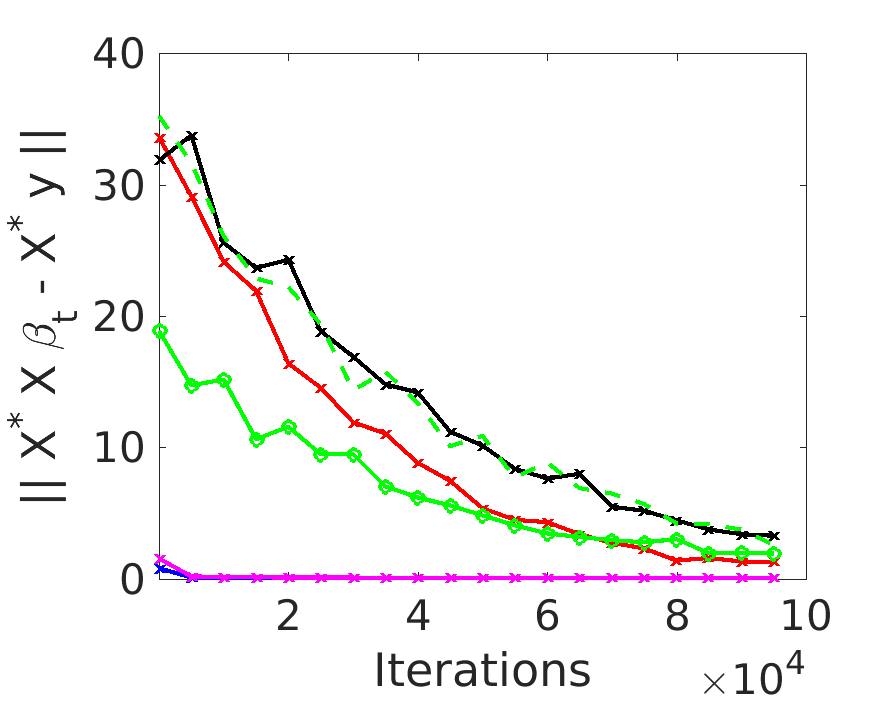}
& \includegraphics[scale=0.35]{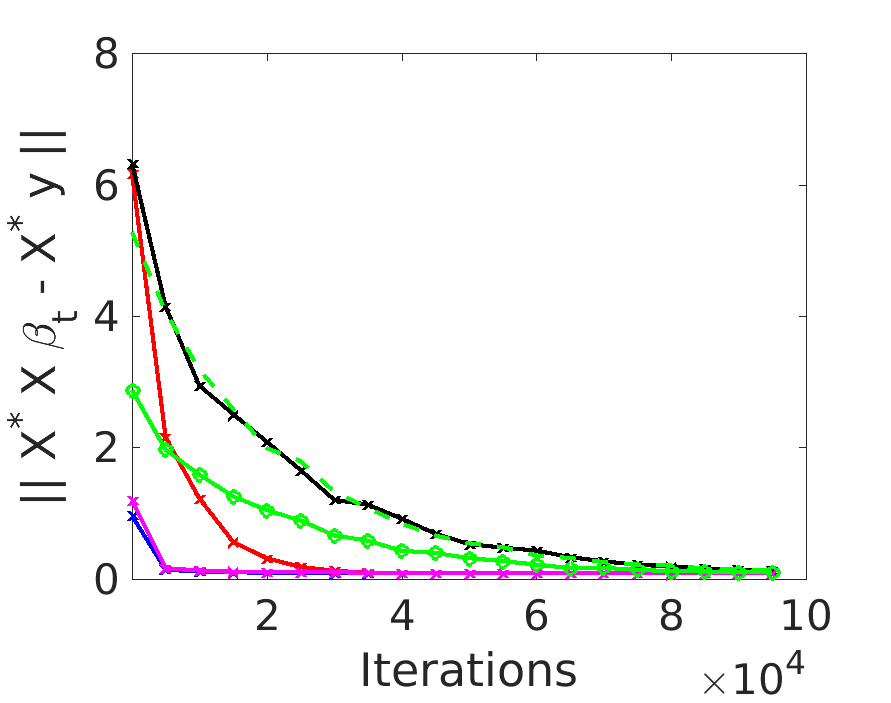} \\
\end{tabular}
\caption{Simulation results for $\lambda = 10^{-3}$: Error $\|X^*X \bbeta_t - X^*y\|$ versus iteration count.}
\label{fig:xnrm}
\end{figure}

%

\section{Conclusion}\label{sec:conclude}
This work extends the parallel analysis of the randomized Kaczmarz (RK) and randomized Gauss-Seidel (RGS) methods to the setting of ridge regression.  By presenting a parallel study of the behavior of these two methods in this setting, comparisons and connections can be made between the approaches as well as other existing approaches.  In particular, we demonstrate that the augmented projection approach of Ivanov and Zhdanov performs a mix of RK and RGS style updates in such a way that many iterations yield no progress.  Motivated by this framework, we present a new approach which eliminates this drawback, and provide an analysis demonstrating that the RGS variant is preferred in the overdetermined case while RK is preferred in the underdetermined case.  This extends previous analysis of these types of iterative methods in the classical ordinary least squares setting, which are highly suboptimal if directly applied to the setting of ridge regression.

\section*{Acknowledgments}
Needell was partially supported by NSF CAREER grant $\#1348721$, and the Alfred P. Sloan Fellowship. Ramdas was supported in part by ONR MURI grant N000140911052. The authors would also like to thank the Institute of Pure and Applied Mathematics (IPAM) at which this collaboration started. 

\bibliographystyle{agsm}
\bibliography{rk}


\end{document}